\documentclass[10pt]{article}

\usepackage{amsmath,amsthm,amssymb}
\usepackage{a4wide}
\usepackage{cite}

\newtheorem{theorem}{Theorem}[section]
\newtheorem{lemma}[theorem]{Lemma}
\newtheorem{corollary}[theorem]{Corollary}
\theoremstyle{definition}
\newtheorem{definition}[theorem]{Definition}
\theoremstyle{remark}
\newtheorem{remark}[theorem]{Remark}

\begin{document}

\title{Dynamic Hardy type inequalities via alpha-conformable derivatives on time scales}

\author{Ahmed A. El-Deeb$^{1}$\\
{\tt \small ahmedeldeeb@azhar.edu.eg}
\and Samer D. Makharesh$^{1}$\\
{\tt \small sameeermakarish@yahoo.com}
\and Delfim F. M. Torres$^{2}$\\
{\tt \small delfim@ua.pt}}

\date{$^{1}$Department of Mathematics, Faculty of Science,\\ 
Al-Azhar University, 11884 Nasr City, Cairo, Egypt\\[0.3cm]
${^2}$\mbox{Center for Research and Development in Mathematics and Applications (CIDMA),}
Department of Mathematics, University of Aveiro, 3810-193 Aveiro, Portugal}

\maketitle


\begin{abstract}
We prove new Hardy-type $\alpha$-conformable dynamic inequalities on time scales. 
Our results are proved by using Keller's chain rule, the integration by parts formula, 
and the dynamic H\"{o}lder inequality on time scales. When $\alpha=1$, then we obtain 
some well-known time-scale inequalities due to Hardy. As special cases, we obtain new 
continuous, discrete, and quantum inequalities. 

\medskip


\end{abstract}


\section{Introduction}

Hardy (1877--1947) established in 1920 a now classical discrete inequality.

\begin{theorem}[See \cite{b2}]
\label{thm:1.1}	
Consider the nonnegative sequence of real numbers 
$\{\varrho(\imath)\}_{\imath=1}^\infty$. 
For $p > 1$, one has
\begin{equation}
\label{a1}
\sum_{\imath=1}^{\infty}\frac{1}{\imath^p}\left(\sum_{\jmath=1}^{\imath}
\varrho(\jmath)\right)^p\leq \left(\frac{p}{p-1}\right)^p
\sum_{\imath=1}^{\infty}\varrho^p(\imath).
\end{equation}
\end{theorem}

In 1928, by using the calculus of variations, Hardy 
introduced the continuous version of inequality \eqref{a1}.

\begin{theorem}[See \cite{b3}]
\label{thm:1.2}	
Let $\eta$ be a nonnegative continuous function on $[0,\infty)$. 
If $p > 1$, then 
\begin{equation}\label{a2}
\int_{0}^{\infty}\frac{1}{\pi^p}\left(\int_{0}^{\pi}\eta(s)ds\right)^p d\pi
\leq \left(\frac{p}{p-1}\right)^p\int_{0}^{\infty}\eta^p(\pi)d\pi.
\end{equation}
Moreover, the constant $\left(\frac{p}{p-1}\right)^p$ in \eqref{a2} is sharp.
\end{theorem}

In 1927, Hardy and Littlewood (1885--1977) proved a discrete inequality 
that is an extension of \eqref{a1}.

\begin{theorem}[See \cite{b4}]
\label{thm:1.3}
Consider the sequence of nonnegative real numbers $\{\varrho(\imath)\}_{\imath=1}^\infty$.
\begin{description}
\item[($i$)] For $p > 1$ and $\alpha > 1$, one has
\begin{equation} 
\label{a3}
\sum_{\imath=1}^{\infty}\frac{1}{\imath^\alpha}\bigg(
\sum_{\jmath=1}^{\imath}\varrho(\jmath)\bigg)^p 
\leq \c{C}(\alpha,p)\sum_{\imath=1}^{\infty}
\frac{1}{\imath^{\alpha-p}}\varrho^p(\imath);
\end{equation}
\item[($ii$)] For $p>1$ and $\alpha < 1$, one has
\begin{equation}  
\label{a4}
\sum_{\imath=1}^{\infty}\frac{1}{\imath^\alpha}\bigg(\sum_{\jmath=\imath}^{\infty}
\varrho(\jmath)\bigg)^p \leq \c{C}(\alpha,p)\sum_{\imath=1}^{\infty}
\frac{1}{\imath^{\alpha-p}} \varrho^p(\imath);
\end{equation}
\end{description}
where the $\c{C}(\alpha,p)\geq0$ in the inequalities \eqref{a3}
and \eqref{a4} depend on $\alpha$ and $p$.
\end{theorem}

In \cite{b4}, the authors also studied the continuous analogous
of Theorem~\ref{thm:1.3}.

\begin{theorem}[See \cite{b4}]
\label{thm:1.4}	
Let $\eta$ be a nonnegative continuous function on $[0,\infty)$. 
If $p > 1$, then 
\begin{equation*}
\int_{0}^{\infty}\bigg(\frac{1}{\pi}\int_{\pi}^{\infty}\eta(s)ds\bigg)^p d\pi \leq
p^p\int_{0}^{\infty}\eta^p(\pi)d\pi
\end{equation*}
or, by a trivial transformation,
\begin{equation}  
\label{a11}
\int_{0}^{\infty}\bigg(\int_{\pi}^{\infty}\eta(s)ds\bigg)^p d\pi \leq
p^p\int_{0}^{\infty}\pi^p\eta^p(\pi)d\pi.
\end{equation}
\end{theorem}

Hardy studied the integral form of \eqref{a3} and \eqref{a4} as follows.

\begin{theorem}[See \cite{b3}]
Consider the continuous function $\eta\geq0$ on $[0,\infty)$.
\begin{description}
\item[($i$)] For $\alpha>1$ and $p>1$,  we have
\begin{equation}\label{a5}
\int_{0}^{\infty}\frac{1}{\pi^\alpha}\bigg(\int_{0}^{\pi}\eta(s)ds\bigg)^pd\pi 
\leq \Big(\frac{p}{\alpha-1}\Big)^p\int_{0}^{\infty}\frac{1}{\pi^{\alpha-p}}\eta^p(\pi)d\pi.
\end{equation}

\item[($ii$)] For $\alpha<1$ and $p>1$, we have
\begin{equation}
\label{a6}
\int_{0}^{\infty}\frac{1}{\pi^\alpha}\bigg(\int_{\pi}^{\infty}\eta(s)ds\bigg)^pd\pi 
\leq \Big(\frac{p}{1-\alpha}\Big)^p\int_{0}^{\infty}\frac{1}{\pi^{\alpha-p}}\eta^p(\pi)d\pi.
\end{equation}
\end{description}
\end{theorem}

In the same year of 1928, Copson also extended \eqref{a1}.

\begin{theorem}[See \cite{b12}]
Consider the nonnegative sequences $\{\varrho(\imath)\}_{\imath=1}^\infty$ 
and $\{\varsigma(\imath)\}_{\imath=1}^\infty$ of real numbers. Then,
\begin{equation}
\label{h20}
\sum_{\imath=1}^{\infty}\frac{\varsigma(\imath)\Big(\sum_{\jmath=1}^{\imath}
\varsigma(\jmath)\varrho(\jmath)\Big)^p}{\Big(\sum_{\jmath=1}^{\imath}
\varsigma(\jmath)\Big)^\alpha}\leq \Big(\frac{p}{\alpha-1}\Big)^p
\sum_{\imath=1}^{\infty}\varsigma(\imath)\varrho^p(\imath)\Big(\sum_{\jmath=1}^{\imath}
\varsigma(\jmath)\Big)^{p-\alpha},
\qquad \text{for} \quad p\geq\alpha>1,
 \end{equation}
and
\begin{equation}
\label{h21}
\sum_{\imath=1}^{\infty}\frac{\varsigma(\imath)\Big(\sum_{\jmath=\imath}^{\infty}
\varsigma(\jmath)\varrho(\jmath)\Big)^p}{\Big(\sum_{\jmath=1}^{\imath}g(\jmath)\Big)^\alpha}
\leq \Big(\frac{p}{1-\alpha}\Big)^p\sum_{\imath=1}^{\infty}\varsigma(\imath)
\varrho^p(\imath)\Big(\sum_{\jmath=1}^{\imath}\varsigma(\jmath)
\Big)^{p-\alpha},
\qquad \text{for} \quad p>1>\alpha\geq0.
 \end{equation}
\end{theorem}

In 1970, Leindler discussed the result \eqref{h20} 
when the limit of summation  $\sum_{n=1}^{\infty}r(m)<\infty$ 
changed from $\imath$ to $\infty$.

\begin{theorem}[See \cite{b14}]
\label{leindler}
Consider the nonnegative real numbers sequences 
$\{\varrho(\imath)\}_{\imath=1}^\infty$ 
and $\{\varsigma(\imath)\}_{\imath=1}^\infty$ with
$\sum_{\jmath=\imath}^{\infty}\varsigma(\jmath)<\infty$. 
For $p>1>\alpha\geq0$, one has
\begin{equation}
\label{h30}
\sum_{\imath=1}^{\infty}\frac{\varsigma(\imath)\Big(
\sum_{\jmath=1}^{\imath} \varsigma(\jmath)
\varrho(\jmath)\Big)^p}{\Big(\sum_{\jmath=\imath}^{\infty}
\varsigma(\jmath)\Big)^\alpha}
\leq \Big(\frac{p}{1-\alpha}\Big)^p\sum_{\imath=1}^{\infty}
\varsigma(\imath)\varrho^p (\imath)\Big(\sum_{\jmath=\imath}^{\infty}
\varsigma(\jmath)\Big)^{p-\alpha}.
\end{equation}
\end{theorem}

Copson investigated the continuous form of \eqref{h20} 
and \eqref{h21} in 1976.

\begin{theorem}[See \cite{copson1976}]
Consider the continuous function $\eta\geq0$ and $\xi$ on $[0,\infty)$. Then,
\begin{equation}
\label{h24}
\int_{0}^{\infty}\frac{\xi(\pi)\Big(\int_{0}^{\pi}\xi(s)
\eta(\pi)ds\Big)^p}{\Big(\int_{0}^{\pi}\xi(s)ds\Big)^\alpha}d\pi
\leq \Big(\frac{p}{\alpha-1}\Big)^p\int_{0}^{\infty}\xi(\pi)
\eta^p(\pi)\bigg(\int_{0}^{\pi}\xi(s)ds\bigg)^{p-\alpha}d\pi
\end{equation}
for $1<\alpha\leq p$ and
\begin{equation}
\label{h25}
\int_{0}^{\infty}\frac{\xi(\pi)\Big(\int_{\pi}^{\infty}\xi(s)
\eta(\pi)ds\Big)^p}{\Big(\int_{0}^{\pi}\xi(s)ds\Big)^\alpha}d\pi
\leq \Big(\frac{p}{1-\alpha}\Big)^p\int_{0}^{\infty}\xi(\pi)
\eta^p(\pi)\bigg(\int_{0}^{\pi}\xi(s)ds\bigg)^{p-\alpha}d\pi
\end{equation}
for $0<\alpha\leq 1 < p$.
\end{theorem}

In 1987 Bennett, similarly to what Leindler did in Theorem~\ref{leindler}, 
proved the following result.

\begin{theorem}[See \cite{b9}]
Consider the nonnegative real numbers sequences $\{\varrho(\imath)\}_{\imath=1}^\infty$ 
and $\{\varsigma(\imath)\}_{\imath=1}^\infty$ with
$\sum_{\jmath=\imath}^{\infty}\varsigma(\jmath)<\infty$. 
For $1< \alpha\leq p$, then
\begin{equation}
\label{h31}
\sum_{\imath=1}^{\infty}\frac{\varsigma(\imath)\Big(\sum_{\jmath=\imath}^{\infty}
\varsigma(\jmath)\varrho(\jmath)\Big)^p} {\Big(\sum_{\jmath=\imath}^{\infty}
\varsigma(\jmath)\Big)^\alpha}\leq \Big(\frac{p}{\alpha-1}\Big)^p
\sum_{\imath=1}^{\infty}\varsigma(\imath)\varrho^p(\imath)\Big(
\sum_{\jmath=\imath}^{\infty}\varsigma(\jmath)\Big)^{p-\alpha}.
\end{equation}
\end{theorem}

Over several decades, Hardy-type inequalities have attracted  
many researchers and several refinements and extensions have been 
done to the previous results. We refer the reader to the works 
\cite{b3,b6,b7,b8,b9,b10,b11,b12,b13,b14,b5,b15,b16,b17}, 
and the references cited therein. Here we are particularly
interested in the following extensions proved by Renaud
in 1986.

\begin{theorem}[See \cite{renaud}]
Consider the nonnegative real numbers and nonincreasing sequence 
$\{\varrho(\imath)\}_{\imath=1}^\infty$. For $p>1$, we have
\begin{equation}
\label{h8}
\sum_{\imath=1}^{\infty}\Big(\sum_{\jmath=\imath}^{\infty}
\varrho(\jmath)\Big)^p\geq\sum_{\imath=1}^{\infty}
\imath^p\varrho^p(\imath).
\end{equation}
\end{theorem}

\begin{theorem}[See \cite{renaud}]
Consider a nonnegative and nonincreasing function $\eta$  
on the interval $[0,\infty)$. For $1<p$, we have
\begin{equation}
\label{h9}
\int_{0}^{\infty}\bigg(\int_{\pi}^{\infty}\eta(s)ds\bigg)^pd
\pi\geq\int_{0}^{\infty}\pi^p \eta^p(\pi)d\pi.
\end{equation}
\end{theorem}

\begin{theorem}[See \cite{renaud}]
Consider a nonnegative and nonincreasing function $\eta$  
on the interval $[0,\infty)$. For $p>1$, we have
\begin{equation}\label{h7}
\int_{0}^{\infty}\frac{1}{\pi^p}\Big(\int_{0}^{\pi}\eta(s)ds\Big)^pd\pi
\geq \frac{p}{p-1}\int_{0}^{\infty}\eta^p(\pi)d\pi.
\end{equation}
\end{theorem}

The theory of time scales has become a trend
and is now part of the mathematics subject classification:
see 26E70, for ``Real analysis on time scales'';
34K42, for ``Functional-differential equations on time scales'';
34N05, for ``Dynamic equations on time scales'';
and 35R07, for ``PDEs on time scales''.
The subject has began with the PhD thesis of Hilger, in order to get  
continuous and discrete results together \cite{s9}. 
In books \cite{l3,s5}, Bohner and Peterson introduce 
most basic concepts and definitions related with the theory of time scales. 
In \cite{AMC,deeb119,s5,l6,deeb3,deeb4,p23,p25}, several mathematicians 
investigate new forms of dynamic inequalities. 
\v{R}eh{\'a}k seems to be the first mathematician 
to have introduced a time-scale version of Hardy's inequality,
by obtaining in 2005 a dynamic inequality that unifies inequalities 
\eqref{a1} and \eqref{a2}.

\begin{theorem}[See \cite{b18}]
Let $\mathbb{T}$ be a time scale, 
and $f\in C_{rd}\big([a,\infty)_\mathbb{T},[0,\infty)\big)$. 
If $p>1$, then
\begin{equation}  
\label{a14}
\int_{a}^{\infty}\bigg(\frac{\int_{a}^{\sigma(t)}
\eta(s)\Delta s}{\sigma(t)-a}\bigg)^p\Delta t
< \Big(\frac{p}{p-1}\Big)^p\int_{a}^{\infty}\eta^p(t)\Delta t,
\end{equation}
unless $\eta\equiv0$. Furthermore, 
if $\mu(t)/t\rightarrow0$ as $t\rightarrow\infty$, 
then inequality \eqref{a14} is sharp.
\end{theorem}

Many other dynamic inequalities followed.
For instance, in 2014 Saker et al. established 
the following results on time scales.

\begin{theorem}[See \cite{sk}] 
\label{th2.5}
Let $\mathbb{T}$ be time scale and $ 1\leqslant c \leqslant k$. Let
\begin{eqnarray} 
\label{2.5}
\chi(t)=\int_a^t \lambda(s) \Delta s, 
\ \ \textit{for any}   \ t \in [a,\infty)_\mathbb{T},
\end{eqnarray}
and define
\begin{equation}
\Theta(t)= \int_a^t \lambda(s)\xi(s) \Delta s 
\ \ \textit{for any}   \ t \in [a,\infty)_\mathbb{T}.
\end{equation}
Then,
\begin{eqnarray*}
\int_a^{\infty} \frac{\lambda(t)}{\big(\chi^{\sigma}(t)\big)^c} 
\big(\Theta^\sigma(t)\big)^k  \Delta t \leqslant \frac{k}{c-1} \int_a^{\infty}
\chi^{1-c}(t)\lambda(t)\xi(t) \big(\Theta(t)\big)^{k-1}  \Delta t.
\end{eqnarray*}
and
\begin{eqnarray*}
\int_a^{\infty} \frac{\lambda(t)}{\big(\chi^{\sigma}(t)\big)^c} 
\big(\Theta^\sigma(t)\big)^k  \Delta t 
\leqslant \bigg(\frac{k}{c-1}\bigg)^k \int_a^{\infty}
\dfrac{(\chi^\sigma(t))^{(k-1)c}}{(\chi(t))^{k(c-1)}} 
\lambda(t) \xi^k(t) \Delta t.
\end{eqnarray*}
\end{theorem}

\begin{theorem}[See \cite{sk}] 
\label{th2.6}
Let $\mathbb{T}$ be a time scale and $k>1$ and $0 \leqslant c <1$. 
Let $\chi$ be defined as in \eqref{2.5} and define
\begin{equation*}
\overline{\Theta}(t)= \int_t^\infty \lambda(s)\xi(s) \Delta s 
\ \ \textit{for any}   \ t \in [a,\infty)_\mathbb{T}.
\end{equation*}
Then,
\begin{eqnarray*}
\int_a^{\infty} \frac{\lambda(t)}{\big(\chi^{\sigma}(t)\big)^c} 
\big( \overline{\Theta}(t)\big)^k  \Delta t \leqslant \frac{k}{1-c} \int_a^{\infty}
(\chi^{\sigma}(t))^{1-c}\lambda(t)\xi(t) \big(\overline{\Theta}(t)\big)^{k-1}  \Delta t
\end{eqnarray*}
and
\begin{eqnarray*}
\int_a^{\infty} \frac{\lambda(t)}{\big(\chi^{\sigma}(t)\big)^c} 
\big( \overline{\Theta}(t)\big)^k  \Delta t 
\leqslant \bigg(\frac{k}{1-c}\bigg)^k  \int_a^{\infty}
(\chi^{\sigma}(t))^{k-c} \lambda(t) \xi^k(t)\Delta t.
\end{eqnarray*}
\end{theorem}

In 2015, Saker et al. \cite{copson1} established the following forms of the Hardy-type inequality.

\begin{theorem}[See \cite{copson1}]
Let $\eta$ and $\xi$ be nonnegative rd-continuous functions on $[a,\infty)_\mathbb{T}$ 
with $\mathbb{T}$ a time scale and $a\in[0,\infty)_\mathbb{T}$.
\begin{description}
\item[($i$)] For  $1<\alpha \leq p$, one has
\begin{equation}
\label{saker1}
\int_{a}^{\infty}\frac{\xi(\varpi)\left(
\int_{a}^{\sigma(\varpi)}\xi(\vartheta)\eta(\vartheta)\Delta\vartheta\right)^p}{
\left(\int_{a}^{\sigma(\varpi)}\xi(\vartheta)\Delta\vartheta\right)^\alpha}\Delta\varpi
\leq \left(\frac{p}{\alpha-1}\right)^p\int_{a}^{\infty}
\frac{\xi(\varpi)\eta^p(\varpi)\left(\int_{a}^{\sigma(\varpi)}\xi(\vartheta)
\Delta\vartheta\right)^{\alpha(p-1)}}{\left(\int_{a}^{\varpi}\xi(\vartheta)
\Delta\vartheta\right)^{\alpha(p-1)}}\Delta\varpi.
\end{equation}
\item[($ii$)] If $p>1>\alpha\geq0$, then
\begin{equation}
\label{saker2}
\int_{a}^{\infty}\frac{\xi(\varpi)\Big(\int_{\varpi}^{\infty}
\xi(\vartheta)\eta(\vartheta)\Delta\vartheta\Big)^p}{\Big(
\int_{a}^{\sigma(\varpi)}\xi(\vartheta)\Delta\vartheta\Big)^\alpha}\Delta\varpi
\leq \left(\frac{p}{1-\alpha}\right)^p
\int_{a}^{\infty}\xi(\varpi)\eta^p(\varpi)\bigg(\int_{a}^{\sigma(\varpi)}
\xi(\vartheta)\Delta\vartheta\bigg)^{p-\alpha}\Delta\varpi.
\end{equation}
\item[($iii$)] If $p>1>\alpha\geq0$, then
\begin{equation}
\label{saker3}
\int_{a}^{\infty}\frac{\xi(\varpi)\left(
\int_{a}^{\sigma(\varpi)}\xi(\vartheta)\eta(\vartheta)\Delta\vartheta\right)^p}{
\left(\int_{\varpi}^{\infty}\xi(\vartheta)\Delta\vartheta\right)^\alpha}\Delta\varpi
\leq
\left(\frac{p}{1-\alpha}\right)^p
\int_{a}^{\infty}\xi(\varpi)\eta^p(\varpi)\left(
\int_{\varpi}^{\infty}\xi(\vartheta)\Delta\vartheta\right)^{p-\alpha}\Delta\varpi.
\end{equation}
\item[($iv$)] If $p\geq \alpha>1$, then
\begin{equation}
\label{saker4}
\int_{a}^{\infty}\frac{\xi(\varpi)\left(
\int_{\varpi}^{\infty}\xi(\vartheta)\eta(\vartheta)\Delta\vartheta\right)^p}{
\left(\int_{\varpi}^{\infty}\xi(\vartheta)\Delta\vartheta\right)^\alpha}\Delta\varpi
\leq \left(\frac{p}{\alpha-1}\right)^p
\int_{a}^{\infty}\xi(\varpi)\eta^p(\varpi)\left(
\int_{\varpi}^{\infty}\xi(\vartheta)\Delta\vartheta\right)^{p-\alpha}\Delta\varpi.
\end{equation}
\end{description}
\end{theorem}

Agarwal et al. \cite{agarwal} generalized  inequality \eqref{h7} 
to time scales as follows: for $p>1$,
\begin{equation}
\label{h11}
\int_{0}^{\infty}\frac{1}{t^p}\Big(\int_{0}^{t}\eta(s)\Delta s\Big)^p \Delta t
\geq \frac{p}{p-1}\int_{0}^{\infty}\eta^p(t)\Delta t.
\end{equation}

Recently, in 2020, Saker \cite{sk3} proved the following theorem.

\begin{theorem} 
\label{th2.9}
Assume that $\mathbb{T}$ is a time scale with $\omega \in (0,\infty)_{\mathbb{T}}$. 
If $m\leqslant 0< h< 1$, $\chi(t)= \int_t^\infty \lambda(s) \Delta s$ 
and $\Theta(t)= \int_{\omega}^{t} \lambda(s)\xi(s) \Delta s$, then
\begin{eqnarray*}
\int_{\omega}^\infty \dfrac{\lambda(t)}{\chi^m(t)} (\Theta^{\sigma}(t))^h  \Delta t 
\geqslant \left(\frac{h}{1-m}\right)^h 
\int_{\omega}^{\infty} \lambda(t) \xi^h(t) \chi^{h-m}(t) \Delta t.
\end{eqnarray*}
If $0<h<1<m$, $\chi(t)=\int_t^\infty \lambda(s) \Delta s$ and 
$\overline{\Theta}(t)=\int_{t}^{\infty} \lambda(s)\xi(s) \Delta s$, then
\begin{eqnarray*}
\int_{\omega}^\infty \dfrac{\lambda(t)}{\chi^m(t)} 
\left(\overline{\Theta}(t)\right)^h \Delta t 
\geqslant \left(\frac{hM^m}{m-1}\right)^h  
\int_{\omega}^{\infty} \lambda(t) \xi^h(t) \chi^{h-m}(t) \Delta t,
\end{eqnarray*}
where 
$$ 
M:= \inf_{t \in \mathbb{T}} \frac{\chi^{\sigma}(t)}{\chi(t)} >0.
$$
\end{theorem}

Also in 2020, El-Deeb et  al. \cite{sdm1} established 
a generalization of \eqref{h11} that unifies \eqref{h8} 
and \eqref{h9}: for $p\geq 1$ and $\gamma > 1$, the inequality
\begin{equation}
\label{15}
\int_{a}^{\infty}\frac{\tilde{\lambda} (\zeta)
\breve{\Psi}^p(\zeta)}{\Tilde{\Lambda} ^{\hat{\gamma}} (\zeta)}\Delta \zeta
\geq \frac{p}{\hat{\gamma} -1}
\int_{a}^{\infty}\tilde{\lambda} 
(\zeta)\Tilde{\Lambda}^{p-\hat{\gamma} }(\zeta)\eta^p(\zeta)\Delta \zeta
\end{equation}
holds where
\begin{equation*}
\breve{\Psi}(\zeta)=\int_{a}^{\zeta}\tilde{\lambda} (\eta)\eta(s)\Delta \eta
\qquad \text{and} \qquad 
\Tilde{\Lambda} (\zeta)=\int_{a}^{\zeta}\tilde{\lambda} (\eta)\Delta \eta.
\end{equation*}

Furthermore, El-Deeb et al. \cite{dd1} established a generalization 
of inequalities \eqref{saker1}, \eqref{saker2}, \eqref{saker3} and 
\eqref{saker4} on time scales as follows.

\begin{theorem}[See \cite{dd1}]
\label{thmm1}
Let $\mathbb{T}$ be a time scale with $a\in[0,\infty)_\mathbb{T}$. 
In addition, let $f$, $g$, $k$, $r$, $w$ and $v$ be nonnegative 
rd-continuous functions on $[a,\infty)_\mathbb{T}$ such that $k$ 
is nonincreasing. Assume there exist $\theta,\beta\geq0$ such that 
$\displaystyle\frac{w^\Delta(t)}{w(t)}\leq \theta \Big(\frac{G^\Delta(t)}{G^\sigma(t)}\Big)$ 
and $\displaystyle\frac{v^\Delta(t)}{v^\sigma(t)}\leq \beta \Big(\frac{K^\Delta(t)}{K(t)}\Big)$, 
where
\begin{equation*}
G(t)=\int_{a}^{t}g(s)\Delta s \quad with \quad G(\infty)=\infty 
\quad  and \quad K(t)=\int_{a}^{t}r(s)f(s)\Delta s, \quad t\in[a,\infty)_\mathbb{T}.
\end{equation*}
If $p\geq1$ and $\alpha>\theta+1$, then
\begin{multline}
\label{eqq1}
\int_{a}^{\infty}k^\sigma(t)v^\sigma(t)w(t)g(t)\big(
G^\sigma(t)\big)^{-\alpha}\big(K^\sigma(t)\big)^p\Delta t\\
\leq\bigg(\frac{p+\beta}{\alpha-\theta-1}\bigg)^p
\int_{a}^{\infty}\frac{k^\sigma(t)v^\sigma(t)w(t)r^p(t)f^p(t)\big(
G^\sigma(t)\big)^{\alpha(p-1)}}{g^{p-1}(t)G^{p(\alpha-1)}(t)}\Delta t.
\end{multline}
\end{theorem}

For more results on Hardy-type inequalities 
on time scales we refer to \cite{h3,h5,h6,h10,saker}
and references therein. Here we are interested 
in such inequalities in the fractional sense.

Fractional calculus theory has an important role in mathematical analysis
and applications. Fractional calculus (FC), the theory of integrals 
and derivatives of noninteger order, is a field of research 
with a history dating back to Abel, Riemann and Liouville: 
see \cite{miller1993introduction} for an historical account. 
The most famous and extensively studied fractional operator is given by
\begin{equation*}
I_{a+}^{\alpha}\eta(t)
=\frac{1}{\Gamma(\alpha)}\int_{a}^{x}(x-t)^{\alpha-1}\eta(t)dt,
\end{equation*}
which is called the Riemann--Liouville fractional integral, in honor
of Riemann (1826--1866) and Liouville (1809--1882). The corresponding 
fractional derivative is obtained by composition of the
fractional integral with an integer order derivative.

The definitions of fractional integrals and derivatives are not unique, 
and many definitions of fractional derivative operators  
were introduced and successfully applied to solve complex systems 
in science and engineering: see \cite{daftardar2007analysis,kilbas2006theory,podlubny1998fractional}. 
The study of fractional dynamic equations is nowadays widespread around the world, being useful 
in pure and applied mathematics, physics, engineering, biology, economics, etc. 
They use an integral in their formulation, especially Cauchy's integral formula with its modifications. 
Therefore, they involve difficult calculations. It is well-known that Riemann--Liouville and Caputo 
fractional derivatives do not satisfy the usual derivative rules for the product, 
quotient and chain rules. Moreover, the mean value theorem and Rolle's theorem are not valid 
for the definitions of Riemann--Liouville and Caputo fractional derivatives.

Recently, just based on the classical limit definition of derivative, 
Khalil et al. \cite{khalil2014new} proposed a much simpler definition 
of a fractional derivative of a function $f:\mathbb{R}^{+}\rightarrow \mathbb{R}$,
called the conformable derivative $T_{\alpha}f(t)$, $\alpha\in(0,1]$, defined by
\begin{equation*}
T_{\alpha}f(t)=\lim_{\epsilon\rightarrow 0}
\frac{f(t+\epsilon t^{1-\alpha})-f(t)}{\epsilon}
\end{equation*}
for all $t>0$. This definition found wide resonance in the scientific
community interested in fractional calculus, due to the fact that calculating 
the derivative by this definition is trivial compared with the definitions 
that are based on integration. The researchers in \cite{khalil2014new}
also suggested a definition for the $\alpha$-conformable integral of a
function $\eta$ as follows:
\begin{equation*}
\int_{a}^{b}\eta(t)d_{\alpha}t=\int_{a}^{b}\eta(t)t^{\alpha-1}dt.
\end{equation*}
After the seminal paper \cite{khalil2014new}, Abdeljawad \cite{abdeljawad2015conformable} 
made an extensive research of the newly introduced conformable calculus. In his work, 
he generalizes the definition of conformable derivative $T_{\alpha}^{a}f(t)$ 
of $f:\mathbb{R}^{+}\rightarrow \mathbb{R}$ for $t>a\in \mathbb{R}^{+}$ as
\begin{equation*}
T_{\alpha}^{a}f(t)=\lim_{\epsilon\rightarrow 0}
\frac{f(t+\epsilon (t-a)^{1-\alpha})-f(t)}{\epsilon}.
\end{equation*}
Benkhettou et al. \cite{sdm31} introduced a conformable calculus on an
arbitrary time scale, which is a natural extension of the conformable calculus.
Based on such results, in the last few years many authors pointed out that derivatives 
and integrals of non-integer order are very suitable for the description of properties 
of various real materials, e.g. polymers. Fractional derivatives provide 
an excellent instrument for the description of memory and hereditary properties 
of various materials and processes. This seems to be the main advantages 
of fractional derivatives in comparison with classical integer-order models.

By using the conformable fractional calculus, many inequalities 
have been investigated like Hardy's \cite{sk2,zk}, 
Hermite--Hadamard's \cite{chu2017inequalities,khan2018hermite,set2017hermite}, 
Opial's \cite{sarikaya2017new,sarikayaopial} and Steffensen's inequalities 
\cite{sarikaya2017steffensen}. For example, in 2020, Saker et al. \cite{sk2} 
proved a $\alpha$-conformable version of Theorems~\ref{th2.5} and \ref{th2.6} 
on time scales as follows.

\begin{theorem}[See \cite{sk2}]
Let $\mathbb{T}$ be a time scale and $ 1\leqslant c \leqslant k$. Define
\begin{eqnarray*}
\chi(x)=\int_a^x \lambda(s) \Delta_\alpha s, \ \ \textit{and} \ \ 
\Theta(x)= \int_a^x \lambda(s)\xi(s) \Delta_\alpha s.
\end{eqnarray*}
If 
$$ 
\Theta(\infty) < \infty,\ \ \textit{and} \ \ 
\int_a^{\infty} \frac{\lambda(s)}{\big(\chi^{\sigma}(s)\big)^{c-\alpha+1}}  
\Delta_{\alpha} s < \infty, 
$$
then
\begin{eqnarray*}
\int_a^{\infty} \frac{\lambda(x)}{\big(\chi^{\sigma}(x)\big)^{c-\alpha+1}} 
\big(\Theta(x)\big)^k  \Delta_\alpha x 
\leqslant \bigg(\frac{k}{c-\alpha}\bigg)^k 
\int_a^{\infty}  
\dfrac{\lambda(x) (\chi(x))^{K(\alpha-c)}}{(\chi^{\sigma}(x))^{(1-k)(c-\alpha+1)}}  
\xi^k(x) \Delta_\alpha x.
\end{eqnarray*}
\end{theorem}

\begin{theorem}[See \cite{sk2}]
Let $\mathbb{T}$ be a time scale, $0\leqslant c < 1$ and $k>1$. Define
\begin{eqnarray*}
\chi(x)=\int_a^x \lambda(s) \Delta_\alpha s 
\ \ \textit{and}  \ \ 
\Theta(x)= \int_x^\infty \lambda(s)\xi(s) \Delta_\alpha s.
\end{eqnarray*}
If 
$$ 
\Theta(\infty) < \infty
\ \ \textit{and} \ \ 
\int_a^{\infty} \frac{\lambda(s)}{\big(\chi^{\sigma}(s)\big)^{c-\alpha+1}}  
\Delta_{\alpha} s < \infty, 
$$
then
\begin{eqnarray*}
\int_a^{\infty} \frac{\lambda(x)}{\big(\chi^{\sigma}(x)\big)^{c-\alpha+1}} 
\big(\Theta^\sigma(x)\big)^k  \Delta_\alpha x \leqslant \bigg(
\frac{k}{c-\alpha}\bigg)^k \int_a^{\infty} 
(\chi^{\sigma}(x))^{k-c+\alpha-1} \lambda(x)\xi^k(x) \Delta_\alpha x.
\end{eqnarray*}
\end{theorem}

In 2021, Zakarya et al. \cite{zk} obtained $\alpha$-conformable versions 
on time scales of Theorem~\ref{th2.9}.

\begin{theorem}[See \cite{zk}]
Assume that $\mathbb{T}$ is a time scale with $\omega \in (0,\infty)_{\mathbb{T}}$. 
Let $k\leqslant 0< h< 1$, $\alpha \in (0,1]$, and define
\begin{eqnarray*}
\chi(t)= \int_t^\infty \lambda(s) \Delta_\alpha s 
\ \ \textit{and}  \ \ 
\Theta(t)= \int_{\omega}^{t} \lambda(s)\xi(s) \Delta_\alpha s.
\end{eqnarray*}
Then,
\begin{eqnarray*}
\int_{\omega}^\infty \dfrac{\lambda(t)}{\chi^{k-\alpha+1}(t)} 
(\Theta^{\sigma}(t))^h \Delta_\alpha t 
\geqslant \bigg(\frac{h}{\alpha-m}\bigg)^h 
\int_{\omega}^{\infty} \lambda(t) \xi^h(t) \chi^{h-m+\alpha-1}(t) \Delta_\alpha t.
\end{eqnarray*}
\end{theorem}

\begin{theorem}[See \cite{zk}]
Assume that $\mathbb{T}$ is a time scale with $\omega \in (0,\infty)_{\mathbb{T}}$,  
$0<h< 1< k$ and $\alpha \in (0,1]$. Define
\begin{eqnarray*}
\chi(t)= \int_t^\infty \lambda(s) \Delta_\alpha s 
\ \ \textit{and}  \ \ 
\overline{\Theta}(t)= \int_{t}^{\infty} \lambda(s)\xi(s) \Delta_\alpha s
\end{eqnarray*}
such that 
$$ 
M:= \inf_{t \in \mathbb{T}} \frac{\chi^{\sigma}(t)}{\chi(t)} >0.
$$
Then,
\begin{eqnarray*}
\int_{\omega}^\infty \dfrac{\lambda(t)}{\chi^{k-\alpha+1}(t)} 
(\overline{\Theta}(t))^h \Delta_\alpha t 
\geqslant \bigg(\frac{hM^{k-\alpha+1}}{k-\alpha}\bigg)^h  
\int_{\omega}^{\infty} \lambda(t) \xi^h(t) \chi^{h-k+\alpha-1}(t) \Delta_\alpha t.
\end{eqnarray*}
\end{theorem}

Here, we prove new Hardy-type dynamic inequalities 
via the $\alpha$-conformable calculus on time scales.
Our inequalities have a completely new form and may be 
considered as extensions of inequalities \eqref{saker1}, \eqref{saker2}, 
\eqref{saker3} and \eqref{saker4}. As special cases, we obtain some new continuous, 
discrete and quantum inequalities of Hardy-type, generalizing those 
obtained in the literature.

The paper is organized as follows. In Section~\ref{sec:2},
we briefly recall necessary results and notions;
the original results being then given and proved in 
Section~\ref{sec:3}. We end with Section~\ref{sec:4}
of conclusion.


\section{Preliminaries}
\label{sec:2}

We recall the necessary definitions and concepts 
about the time-scale $\alpha$-conformable calculi, 
which are used in the next section. 
For more details we refer the readers to \cite{s5,l3,sdm31,sdm32}. 

Every nonempty arbitrary closed subset of the real numbers 
is called a time scale, being denoted by $\mathbb{T}$. One assumes that 
$\mathbb{T}$ has the standard topology on the real numbers $\mathbb{R}$. 
The forward jump operator $\sigma: \mathbb{T}\rightarrow \mathbb{T}$
is defined by
\begin{equation}  
\label{sigma}
\sigma(t):=\inf\{s\in \mathbb{T} : s>t\}, \qquad t\in\mathbb{T},
\end{equation}
and the backward jump operator $\rho: \mathbb{T}:\rightarrow \mathbb{T}$ by
\begin{equation}  
\label{nu}
\rho(t):=\sup\{s\in \mathbb{T} : s<t\}, \qquad t\in\mathbb{T}.
\end{equation}

In  Definitions~\ref{sigma} and \ref{nu} 
we set $\sup\mathbb{T}=\inf\emptyset$ (i.e., $\sigma(t)=t$ 
if $t$ is the minimum of $\mathbb{T}$) and $\inf\mathbb{T}=\sup\emptyset$ 
(i.e., $\rho(t)=t$ if $t$ is the maximum),
where $\emptyset$ is the empty set.

\begin{definition}[See \cite{sdm31}] 
\label{def1}
Let  $\eta : \mathbb{T} \rightarrow \mathbb{R}$, 
$t \in {\mathbb{T}}^k$, and $\alpha \in (0,1].$ 
For $t >0,$ we define $T^{\Delta}_{\alpha}(\eta)(t)$ 
to be the number (provided it exists) with the property  
that, given any $\epsilon > 0,$ there is a $\delta$-neighborhood 
$U_t \subset \mathbb{T}$ of $t,$  $\delta >0,$ such that
\begin{equation*}
\left|[\eta(\sigma(t))-\eta(s)]t^{1-\alpha}
-T^{\Delta}_{\alpha}(\eta)(t)[\sigma(t)-s]\right|
\leq\varepsilon|\sigma(t)-s|
\end{equation*}
for all $s \in U_t$. We call $T^{\Delta}_{\alpha}(\eta)(t)$ the conformable  
fractional derivative of $\eta$ of order $\alpha$ at $t$,  
and the conformable fractional derivative on $\mathbb{T}$ at $0$ is defined 
as $T^{\Delta}_{\alpha}(\eta)(0)=\lim_{t \longrightarrow 0+} T^{\Delta}_{\alpha}(\eta)(t).$
\end{definition}

\begin{lemma}[See \cite{sdm31}] 
Let $\alpha \in (0,1].$ Suppose $\alpha$-conformable differentiable 
of order $\alpha$ at $\zeta \in {\mathbb{T}}^k,$ and continuous 
function $\xi:\mathbb{T}\rightarrow\mathbb{R}$ and the 
differentiable continuously function $\eta:\mathbb{R}\rightarrow\mathbb{R}$. 
There exists a constant $c\in[\zeta,\sigma(\zeta)]_{\mathbb{R}}$ such that
\begin{equation}
\label{chain1}
T^{\Delta}_{\alpha}(\eta\circ \xi)(\zeta)=\eta'(\xi(c))T^{\Delta}_{\alpha}(\xi)(\zeta).
\end{equation}
\end{lemma}

\begin{lemma}[See \cite{sdm31}]
Let $\eta:\mathbb{R}\rightarrow\mathbb{R}$ be  
continuously  differentiable, $\alpha \in (0,1]$ 
and $\xi:\mathbb{T}\rightarrow\mathbb{R}$ 
be a $\alpha$-conformable differentiable function. 
Then $(\eta\circ \xi):\mathbb{T}\rightarrow\mathbb{R}$ 
is $\alpha$-conformable differentiable and we have
\begin{equation}
\label{chain2}
T^{\Delta}_{\alpha} (\eta\circ \xi)(t)
=\bigg\{\int_{0}^{1}\eta'\big(\xi(t)
+h\mu(t)t^{\alpha-1}T^{\Delta}_{\alpha} (\xi(t))\big)dh\bigg\} 
T^{\Delta}_{\alpha}(\xi)(t).
\end{equation}
\end{lemma}

For the continuous functions $\eta$ and $\xi$, 
we have that the product  $\eta \xi:\mathbb{T}
\longrightarrow \mathbb{R}$ is conformable 
fractional differentiable with
\begin{equation} 
\label{product}
T^{\Delta}_{\alpha}(\eta \xi)= T^{\Delta}_{\alpha}(\eta)\xi 
+\eta^{\sigma} T^{\Delta}_{\alpha} (\xi)
=T^{\Delta}_{\alpha}(\eta)\xi^{\sigma} +\eta T^{\Delta}_{\alpha} (\xi).
\end{equation}

The $\alpha$-conformable integration by parts formula on time scales 
is given in the following Lemma.

\begin{lemma}[See \cite{sdm31}]
Suppose that $a,$ $b \in \mathbb{T}$ where $b>a.$ If $\eta,$ $\xi$ 
are conformable $\alpha$-fractional differentiable and $\alpha \in (0,1],$ then
\begin{equation}
\label{parts}
\int_{a}^{b} \eta(t) T_\alpha^{\Delta} \xi(t)\Delta_\alpha t
=\Big[\eta(t)\xi(t)\Big]_a^b-\int_{a}^{b}T_\alpha^{\Delta} 
\eta(t) \xi^\sigma(t)\Delta_\alpha t.
\end{equation}
\end{lemma}

\begin{lemma}[The $\alpha$-conformable H\"{o}lder inequality, see \cite{dd2}]
Let $a,b\in\mathbb{T}$ with $a < b$. If $\alpha \in (0,1]$ 
and $\eta,\xi: \mathbb{T}\longrightarrow \mathbb{R},$  then
\begin{equation}  
\label{holder}
\int_{a}^{b}|\eta(t)\xi(t)|\Delta_\alpha t
\leq \bigg(\int_{a}^{b}\eta^p(t)\Delta_\alpha t\bigg)^{1/p}
\bigg(\int_{a}^{b} \xi^q(t)\Delta_\alpha t\bigg)^{1/q}
\end{equation}
where $p,q>1$ and $1/p+1/q=1$.
\end{lemma}

We need relations between different types of calculus on general
time scales $\mathbb{T}$ and for the particular cases of continuous, 
discrete, and quantum calculi. Such relations are found in \cite{s5,l3,sdm31,sdm32}:
\begin{description}
\item[(i)] for any time scale $\mathbb{T}$, we have
\begin{equation*}
\begin{gathered}
(\eta)^{\Delta_{\alpha}}(t)=(\eta)^{\Delta}(t)t^{1-\alpha},\\
\int_{a}^{b}\eta(t)\Delta_{\alpha}t=\int_{a}^{b}\eta(t)t^{\alpha-1}\Delta t.
\end{gathered}
\end{equation*}

\item[(ii)] If $\mathbb{T}=\mathbb{R}$, then
\begin{equation}
\label{r}
\begin{gathered}
t=\sigma(t),\\
0=\mu(t),\\
\eta^\Delta(t)=\eta'(t),\\
\int_{a}^{b}\eta(t)\Delta t=\int_{a}^{b}\eta(t)dt.
\end{gathered}
\end{equation}

\item[(iii)] If $\mathbb{T}=\mathbb{Z}$, then
\begin{equation}
\label{z}
\begin{gathered}
\sigma(t)=t+1,\\
\mu(t)=1,\\
\eta^\Delta(t)=\Delta \eta(t),\\
\int_{a}^{b}\eta(t)\Delta t=\sum_{t=a}^{b-1}\eta(t).
\end{gathered}
\end{equation}

\item[(iv)] If $\mathbb{T}=h \mathbb{Z}$, then
\begin{equation} 
\label{hz}
\begin{gathered}
\sigma(t)=t+h,\\
\mu(t)=h,\\
\eta^\Delta(t)=\frac{\eta(t+h)-\eta(t)}{h},\\
\int_{a}^{b}\eta(t)\Delta t=\sum_{t=\frac{a}{h}}^{\frac{b}{h}-1}h \eta(ht).
\end{gathered}
\end{equation}

\item[(v)] If $\mathbb{T}=\overline{q^\mathbb{Z}}$, then
\begin{equation}\label{q^z}
\begin{gathered}
\sigma(t)=qt,\\
\mu(t)=(q-1)t,\\
\eta^\Delta(t)=\frac{\eta(qt)-\eta(t)}{(q-1)t},\\
\int_{a}^{b}\eta(t)\Delta t=(q-1)\sum_{t=\log_{q}{a}}^{\log_{q}{b}-1}q^t \eta(q^t).
\end{gathered}
\end{equation}
\end{description}


\section{Main Results}
\label{sec:3}

First, we enlist assumptions for the proofs of our main results.

\begin{description}
\item[($S_1$)] $\mathbb{T}$ is a time scale, 
$p\geq 1$, $a\in[0,\infty)$, and $\alpha\in(0,1]$.

\item[($S_2$)] $f$, $g$, $k$, $r$, $w$ and $v\geq0$ are  
rd-continuous functions on $[a,\infty)$ with $k$ monotonous.

\item[($S_3$)] $\theta$ and $\beta$ are nonnegative constants.

\item[($S_4$)] $\gamma>\theta+1$.

\item[($S_5$)] $0\leq\gamma<\alpha$.

\item[($S_6$)] $G(\varsigma)=\int_{a}^{\varsigma}g(\varsigma)\Delta_{\alpha}\varsigma$, 
$G(\infty)=\infty$, $t\in[a,\infty)$.

\item[($S_7$)] $H(\varsigma)
=\int_{\varsigma}^{\infty}g(\varsigma)\Delta_{\alpha}\varsigma$,  
$t\in[a,\infty)$.

\item[($S_8$)] $K(\varsigma)
=\int_{a}^{\varsigma}r(\varsigma)f(\varsigma)\Delta_{\alpha}\varsigma$, 
$t\in[a,\infty)$.

\item[($S_9$)] $F(\varsigma)=\int_{\varsigma}^{\infty}r(\varsigma)
f(\varsigma)\Delta_{\alpha}\varsigma$, $t\in[a,\infty)$.

\item[($S_{10}$)] $\displaystyle\frac{T_\alpha^{\Delta} w(\varsigma)}{w(\varsigma)}
\leq \theta \Big(\frac{T_\alpha^{\Delta} G(\varsigma)}{G^\sigma(\varsigma)}\Big)$.

\item[($S_{11}$)] $\displaystyle\frac{T_\alpha^{\Delta} w(\varsigma)}{w(\varsigma)}
\leq \theta \Big(\frac{T_\alpha^{\Delta} H(\varsigma)}{H^\sigma(\varsigma)}\Big)$.

\item[($S_{12}$)] $\displaystyle\frac{T_\alpha^{\Delta} w(\varsigma)}{w^{\sigma}(\varsigma)}
\geq \theta \Big(\frac{T_\alpha^{\Delta} G(\varsigma)}{G(\varsigma)}\Big)$.

\item[($S_{13}$)] $\displaystyle\frac{T_\alpha^{\Delta} w(\varsigma)}{w^{\sigma}(\varsigma)}
\geq \theta \Big(\frac{T_\alpha^{\Delta} H(\varsigma)}{H(\varsigma)}\Big)$.

\item[($S_{14}$)] $\displaystyle\frac{T_\alpha^{\Delta} v(\varsigma)}{v^\sigma(\varsigma)}
\leq \beta \Big(\frac{ T_\alpha^{\Delta} K(\varsigma)}{K(\varsigma)}\Big)$.

\item[($S_{15}$)] $\displaystyle\frac{T_\alpha^{\Delta} v(\varsigma)}{v(\varsigma)}
\geq \beta \Big(\frac{ T_\alpha^{\Delta} F(\varsigma)}{F^{\sigma}(\varsigma)}\Big)$.

\item[($S_{16}$)] $k(t)=v(t)=w(t)=1$.

\item[($S_{17}$)] $r(t)=g(t)$.

\item[($S_{18}$)] $r(t)=g(t)=1$.

\item[($S_{19}$)] $\theta=\beta=0$.

\item[($S_{20}$)] $a=0$.

\item[($S_{21}$)] $a=1$.

\item[($S_{22}$)] $\gamma=p$.
\end{description}

Now, we are ready to state and prove our original results, 
which extend several results in the literature.

\begin{theorem}
\label{thm1}
Let $S_1$, $S_2$, $S_3$, $S_4$, $S_6$, $S_8$, $S_{10}$, $S_{14}$ be satisfied. Then,
\begin{multline}  
\label{eq1}
\int_{a}^{\infty}k^\sigma(t)v^\sigma(t)w(t)g(t)\big(G^\sigma(t)\big)^{\alpha
-\gamma-1}\big(K^\sigma(t)\big)^{p-\alpha+1}\Delta_{\alpha} t\\
\leq \bigg(\frac{p+\beta-\alpha+1}{\gamma-\theta-\alpha}\bigg)^p
\int_{a}^{\infty}\frac{k^\sigma(t)v^\sigma(t)w(t)r^p(t)f^p(t)\big(G^\sigma(t)\big)^{(1
-\alpha+\gamma)(p-1)} \big(K^\sigma\big(t)\big)^{1-\alpha} }{g^{p-1}(t)
G^{p(\gamma-\alpha)}(t)}\Delta_{\alpha} t.
\end{multline}
\end{theorem}

\begin{proof}
Applying \eqref{parts} with
\begin{equation*}
T^{\Delta}_{\alpha}u(t)=w(t)g(t)\big(G^\sigma(t)\big)^{\alpha-\gamma-1} 
\qquad \text{and} \qquad 
z^\sigma(t)=k^\sigma(t)v^\sigma(t)\big(K^\sigma(t)\big)^{p-\alpha+1},
\end{equation*}
we have
\begin{multline}
\label{1}
\int_{a}^{\infty}k^\sigma(t)v^\sigma(t)w(t)g(t)\big(G^\sigma(t)\big)^{\alpha
-\gamma-1}\big(K^\sigma(t)\big)^{p-\alpha+1}\Delta_{\alpha} t\\
=\Big[u(t)k(t)v(t)K^{p-\alpha+1}(t)\Big]_a^\infty
+\int_{a}^{\infty}\big(-u(t)\big)T^{\Delta}_{\alpha}  
\Big(k(t)v(t)K^{p-\alpha+1}(t)\Big) \Delta_{\alpha} t,
\end{multline}
where
\begin{equation*}
u(t)=-\int_{t}^{\infty}w(s)g(s)\big(
G^\sigma(s)\big)^{\alpha-\gamma-1}\Delta_{\alpha} s.
\end{equation*}
Using \eqref{chain1}, \eqref{product}, and  $S_{10}$, we have
\begin{eqnarray*}
T^{\Delta}_{\alpha} \Big(w(s)G^{\alpha-\gamma}(s)\Big)
&=&T^{\Delta}_{\alpha} w(s)\big(G^\sigma(s)\big)^{\alpha-\gamma}
+w(s)T^{\Delta}_{\alpha} \big(G^{\alpha-\gamma}(s)\big)\\
&\leq&\theta w(s)T^{\Delta}_{\alpha} G(s)\big(G^\sigma(s)\big)^{\alpha-\gamma-1}
+(\alpha-\gamma)w(s)G^{\alpha-\gamma-1}(c)T^{\Delta}_{\alpha} G(s).
\end{eqnarray*}
Since $T^{\Delta}_{\alpha} G(s)=g(s)\geq0$, $c\leq\sigma(s)$ and $\gamma>1$, we get
\begin{eqnarray*}
T^{\Delta}_{\alpha} \Big(w(s)G^{\alpha-\gamma}(s)\Big)
&\leq&\theta w(s)g(s)\big(G^\sigma(s)\big)^{\alpha-\gamma-1}
+(\alpha-\gamma)w(s)g(s)\big(G^\sigma(s)\big)^{\alpha-\gamma-1}\\
&=&(\alpha-\gamma+\theta)w(s)g(s)\big(G^\sigma(s)\big)^{\alpha-\gamma-1}
\end{eqnarray*}
for $c\in[s,\sigma(s)]$. This gives us that
\begin{equation*}
w(s)g(s)\big(G^\sigma(s)\big)^{\alpha-\gamma-1}
\leq\frac{1}{\alpha-\gamma+\theta} T^{\Delta}_{\alpha}\Big(w(s)G^{\alpha-\gamma}(s)\Big).
\end{equation*}
Hence,
\begin{equation}
\label{2}
\begin{split}
-u(t)=\int_{t}^{\infty}w(s)g(s)\big(G^\sigma(s)\big)^{\alpha-\gamma-1}\Delta_{\alpha} s
&\leq\frac{1}{\alpha-\gamma+\theta}\int_{t}^{\infty} 
T^{\Delta}_{\alpha} \Big(w(s)G^{\alpha-\gamma}(s)\Big)\Delta_{\alpha} s\\
&=\frac{1}{\gamma-\theta-\alpha}w(t)G^{\alpha-\gamma}(t).
\end{split}
\end{equation}
Using  \eqref{product} and  \eqref{chain1}, we have
\begin{eqnarray*}
T^{\Delta}_{\alpha} \Big(k(t)v(t)K^{p-\alpha+1}(t)\Big)
&=& T^{\Delta}_{\alpha} \big(k(t)v(t)\big) K^{p-\alpha+1}(t)+k^\sigma(t)v^\sigma(t)
T^{\Delta}_{\alpha}\big(K^{p-\alpha+1}(t)\big)\\
&=& T^{\Delta}_{\alpha} k(t)v(t)K^{p-\alpha+1}(t)
+k^\sigma(t) T^{\Delta}_{\alpha} v(t)K^{p-\alpha+1}(t)\\
&& \quad \quad+(p-\alpha+1)k^\sigma(t)v^\sigma(t)
K^{p-\alpha}(c) T^{\Delta}_{\alpha} K(t)
\end{eqnarray*}
for $c\in[t,\sigma(t)]$. Since $\sigma(t)\geq c$, $1\leq p$, 
$0\leq T^{\Delta}_{\alpha} k(t)$, $r(t)f(t)=T^{\Delta}_{\alpha} K(t)\geq0$,  
and $S_{14}$, we have
\begin{equation}
\label{4}
\begin{split}
T^{\Delta}_{\alpha} \Big(k(t)v(t)K^p(t)\Big)
&\leq\beta k^\sigma(t)v^\sigma(t)r(t)f(t)K^{p-\alpha}(t)+(p-\alpha+1)
k^\sigma(t)v^\sigma(t)r(t)f(t)\big(K^\sigma(t)\big)^{p-\alpha}\\
&\leq(p+\beta-\alpha+1)k^\sigma(t)v^\sigma(t)r(t)f(t)\big(K^\sigma\big(t))^{p-\alpha}.
\end{split}
\end{equation}
Combining \eqref{1}, \eqref{2} and \eqref{4}, we get ($K(a)=0$ and $u(\infty)=0$) that
\begin{eqnarray*}
&&\int_{a}^{\infty}k^\sigma(t)v^\sigma(t)w(t)g(t)\big(
G^\sigma(t)\big)^{\alpha-\gamma-1}\big(K^\sigma(t)\big)^{p-\alpha+1} \Delta_{\alpha} t\\
&&\leq\frac{p+\beta-\alpha+1}{\gamma-\theta-\alpha}
\int_{a}^{\infty}k^\sigma(t)v^\sigma(t)w(t)r(t)f(t)
G^{\alpha-\gamma}(t)\big(K^\sigma\big(t))^{p-\alpha}\Delta_{\alpha} t
\end{eqnarray*}
or, equivalently,
\begin{eqnarray*}
&&\int_{a}^{\infty}k^\sigma(t)v^\sigma(t)w(t)g(t)\big(
G^\sigma(t)\big)^{\alpha-\gamma-1}\big(K^\sigma(t)\big)^{p-\alpha+1}\Delta_\alpha t\\
&&\leq\frac{p+\beta-\alpha+1}{\gamma-\theta-\alpha}
\int_{a}^{\infty}\bigg(\big(k^\sigma(t)v^\sigma(t)
w(t)g(t)\big)^{(p-1)/p}\big(G^\sigma(t)\big)^{(\alpha-\gamma-1)(p
-1)/p}\big(K^\sigma\big(t))^{(p-1)(p-\alpha+1)/p}\bigg)\\
&&\quad\times\bigg(\frac{\big(k^\sigma(t)v^\sigma(t)w(t)\big)^{1/p}
r(t)f(t)\big(G^\sigma(t)\big)^{(1-\alpha+\gamma)
(p-1)/p} \big(K^\sigma\big(t))^{(1-\alpha)/p} }{g^{(p-1)/p}(t)
G^{\gamma-\alpha}(t)}\bigg)\Delta_\alpha t.
\end{eqnarray*}
Using \eqref{holder} with $p$ and $p/(p-1)$ indexes, we have
\begin{eqnarray*}
&&\int_{a}^{\infty}k^\sigma(t)v^\sigma(t)w(t)g(t)\big(G^\sigma(t)
\big)^{(\alpha-\gamma-1)}\big(K^\sigma(t)\big)^{(p-\alpha+1)}\Delta_{\alpha} t\\
&&\leq\frac{p+\beta-\alpha+1}{\gamma-\theta-\alpha}
\bigg(\int_{a}^{\infty}k^\sigma(t)v^\sigma(t)w(t)g(t)\big(G^\sigma(t)
\big)^{(\alpha-\gamma-1)}\big(K^\sigma\big(t)\big)^{(p-\alpha+1)}\Delta_{\alpha} t\bigg)^{(p-1)/p}\\
&&\quad\times\bigg(\int_{a}^{\infty}\frac{k^\sigma(t)
v^\sigma(t)w(t)r^p(t)f^p(t)\big(G^\sigma(t)\big)^{(1+\gamma-\alpha)(p-1)} 
\big(K^\sigma\big(t)\big)^{1-\alpha} }{g^{p-1}(t)G^{p(\gamma-\alpha)}(t)}\Delta_{\alpha} t\bigg)^{1/p},
\end{eqnarray*}
which implies that
\begin{eqnarray*}
&&\int_{a}^{\infty}k^\sigma(t)v^\sigma(t)w(t)g(t)
\big(G^\sigma(t)\big)^{\alpha-\gamma-1}\big(K^\sigma(t)\big)^{p-\alpha+1}\Delta_{\alpha} t\\
&&\leq\bigg(\frac{p+\beta-\alpha+1}{\gamma-\theta-\alpha}\bigg)^p
\int_{a}^{\infty}\frac{k^\sigma(t)v^\sigma(t)w(t)r^p(t)f^p(t)
\big(G^\sigma(t)\big)^{(1-\alpha+\gamma)(p-1)} 
\big(K^\sigma\big(t)\big)^{1-\alpha} }{g^{p-1}(t)G^{p(\gamma-\alpha)}(t)}\Delta_{\alpha} t.
\end{eqnarray*}
The proof is complete.
\end{proof}	

\begin{remark}
If we take $\alpha=1$ in Theorem~\ref{thm1}, then we recapture Theorem~\ref{thmm1}.
\end{remark}

\begin{corollary} 
\label{1.1}
Theorem~\ref{thm1} with $S_{16}$, $S_{117}$ and $S_{18}$ give us that
\begin{multline} 
\label{1.2}
\int_{a}^{\infty}g(t)\big(G^\sigma(t)\big)^{\alpha
-\gamma-1}\big(K^\sigma(t)\big)^{p-\alpha+1}\Delta_{\alpha} t\\
\leq\bigg(\frac{p-\alpha+1}{\gamma-\alpha}\bigg)^p
\int_{a}^{\infty}\frac{g(t)f^p(t)\big(G^\sigma(t)\big)^{(1-\alpha+\gamma)(p-1)} 
\big(K^\sigma\big(t)\big)^{1-\alpha} }{G^{p(\gamma-\alpha)}(t)}\Delta_{\alpha} t.
\end{multline}
\end{corollary}

\begin{remark}
If we set $\alpha=1$ in Corollary~\ref{1.1}, 
then \eqref{1.2} gives \eqref{saker1}.
\end{remark}

Now, as special cases of our results, we will obtain continuous, 
discrete and quantum $\alpha$-conformable inequalities. This is obtained
by choosing, respectively, the time scales $\mathbb{T}=\mathbb{R}$, 
$\mathbb{T}=h\mathbb{Z}$ and $\mathbb{T}=\mathbb{Z}$, 
and $\mathbb{T}=\overline{q^\mathbb{Z}}$.

\begin{corollary}
\label{cor1}
Putting $\mathbb{T}=\mathbb{R}$ in Theorem~\ref{thm1}, one obtains
from \eqref{r} and \eqref{eq1} that
\begin{multline}
\label{eqcor1}
\int_{a}^{\infty}k(t)v(t)w(t)g(t)G^{\alpha-\gamma-1}(t)K^{p-\alpha+1}(t) t^{\alpha-1} dt\\
\leq\bigg(\frac{p+\beta-\alpha+1}{\gamma-\theta-\alpha}\bigg)^p
\int_{a}^{\infty}\frac{k(t)v(t)w(t)r^p(t)f^p(t)
G^{p-\gamma+\alpha-1}(t) K^{1-\alpha}(t)}{g^{p-1}(t)} t^{\alpha-1} dt,
\end{multline}
where
\begin{equation*}
G(t)=\int_{a}^{t}g(s)s^{\alpha-1} ds 
\quad and \quad  K(t)=\int_{a}^{t}r(s)f(s) s^{\alpha-1} ds.
\end{equation*}
\end{corollary}

\begin{remark}
Corollary~\ref{cor1} with $S_{16}$, $S_{17}$, $S_{19}$ and $S_{20}$ gives
\begin{multline}
\label{eqcor101}
\int_{0}^{\infty}g(t)G^{\alpha-\gamma-1}(t)K^{p-\alpha+1}(t) t^{\alpha-1} dt\\
\leq\bigg(\frac{p-\alpha+1}{\gamma-\alpha}\bigg)^p
\int_{0}^{\infty}g(t)f^p(t)G^{p-\gamma+\alpha-1}(t) K^{1-\alpha}(t) t^{\alpha-1} dt.
\end{multline}
\end{remark}

\begin{remark}
If we set $\alpha=1$ in inequality \eqref{eqcor101}, 
then \eqref{eqcor101} reduces to \eqref{h24}.
\end{remark}

\begin{remark}
If we use assumptions $S_{16}$, $S_{18}$, $S_{19}$ and $S_{20}$
with Corollary~\ref{cor1}, then \eqref{eqcor1} gives
\begin{multline}
\label{eqcor102}
\int_{0}^{\infty}G^{\alpha-\gamma-1}(t)\bigg( 
\int_0^t f(s)s^{\alpha-1} ds \bigg)^{p-\alpha+1} t^{\alpha-1}  dt\\
\leq\bigg(\frac{p-\alpha+1}{\gamma-\alpha}\bigg)^p
\int_{0}^{\infty} G^{p-\gamma+\alpha-1}(t) f^p(t) \bigg( 
\int_0^t f(s)s^{\alpha-1} ds \bigg)^{1-\alpha} t^{\alpha-1} dt.
\end{multline}
\end{remark}

\begin{remark}
If we set $\alpha=1$ in inequality \eqref{eqcor102},  
then \eqref{eqcor102} gives \eqref{a5}.
\end{remark}

\begin{remark}
Under assumptions $S_{16}$, $S_{18}$, $S_{19}$, $S_{20}$ and $S_{22}$, 
the inequality \eqref{eqcor1} of Corollary~\ref{cor1} asserts that
\begin{multline}
\label{eqcor103}
\int_{0}^{\infty} G^{\alpha-\gamma-1}(t)\bigg( 
\int_0^t f(s)s^{\alpha-1} ds \bigg)^{p-\alpha+1} t^{\alpha-1} dt\\
\leq\bigg(\frac{p-\alpha+1}{p-\alpha}\bigg)^p
\int_{0}^{\infty} G^{\alpha-1}(t) f^p(t) \bigg( 
\int_0^t f(s)s^{\alpha-1} ds \bigg)^{1-\alpha} t^{\alpha-1}  dt.
\end{multline}
\end{remark}

\begin{remark}
If we set $\alpha=1$ in inequality \eqref{eqcor103},   
then \eqref{eqcor103} gives  \eqref{a2}.
\end{remark}

\begin{corollary}
\label{cor2}
Putting $\mathbb{T}=h\mathbb{Z}$ in Theorem~\ref{thm1}, inequality \eqref{eq1} gives
\begin{multline}
\sum_{t=\frac{a}{h}}^{\infty}k(ht+h)v(ht+h)w(ht)g(ht)
G^{-\gamma}(ht+h)K^p(ht+h)t^{\alpha-1}\\
\leq\bigg(\frac{p+\beta}{\gamma-\theta-1}\bigg)^p
\sum_{t=\frac{a}{h}}^{\infty}\frac{k(ht+h)v(ht+h)w(ht)r^p(ht)f^p(ht)
G^{\gamma(p-1)}(ht+h)}{g^{p-1}(ht)G^{p(\gamma-1)}(ht)} t^{\alpha-1},
\end{multline}
where
\begin{equation*}
G(t)=h^{\alpha}\sum_{s=\frac{a}{h}}^{\frac{t}{h}-1}g(hs)s^{\alpha-1} 
\quad \text{and} \quad  
K(t)=h^{\alpha}\sum_{s=\frac{a}{h}}^{\frac{t}{h}-1}r(hs)f(hs)s^{\alpha-1}.
\end{equation*}
\end{corollary}

\begin{corollary}
\label{cor3}
If $\mathbb{T}=\mathbb{Z}$ ($h=1$) in Corollary \ref{cor2}, 
then it follows from inequality \eqref{eq1} that
\begin{multline}
\label{eqcor3}
\sum_{t=a}^{\infty}k(t+1)v(t+1)w(t)g(t)G^{-\gamma}(t+1)K^p(t+1)t^{\alpha-1}\\
\leq\bigg(\frac{p+\beta}{\gamma-\theta-1}\bigg)^p
\sum_{t=a}^{\infty}\frac{k(t+1)v(t+1)w(t)r^p(t)f^p(t)
G^{\gamma(p-1)}(t+1)}{g^{p-1}(t)G^{p(\gamma-1)}(t)}t^{\alpha-1},
\end{multline}
where
\begin{equation*}
G(t)=\sum_{s=a}^{t-1}g(s)s^{\alpha-1} 
\quad \text{and} \quad  
K(t)=\sum_{s=a}^{t-1}r(s)f(s)s^{\alpha-1}.
\end{equation*}
\end{corollary}

\begin{remark}
Using assumptions $S_{16}$, $S_{17}$, $S_{19}$ and $S_{21}$,
it follows from \eqref{eqcor3} of Corollary~\ref{cor3} that
\begin{equation}
\sum_{t=1}^{\infty}\frac{g(t)\Big(\sum_{s=1}^{t}
g(s)f(s)\Big)^p}{\Big(\sum_{s=1}^{t}g(s)\Big)^\gamma}
\leq\bigg(\frac{p}{\alpha-1}\bigg)^p
\sum_{t=1}^{\infty}\frac{g(t)f^p(t)\Big(
\sum_{s=1}^{t}g(s)\Big)^{\gamma(p-1)}}{\Big(
\sum_{s=1}^{t-1}g(s)\Big)^{p(\gamma-1)}},
\end{equation}
which is another form of the discrete inequality \eqref{h20}.
\end{remark}

\begin{corollary}
\label{cor8}
Putting $\mathbb{T}=\overline{q^\mathbb{Z}}$ in Theorem~\ref{thm1},
inequality \eqref{eq1} with \eqref{q^z} gives
\begin{multline}
\sum_{t=\log_q a}^{\infty}k(q^{t+1})v(q^{t+1})
w(q^t)g(q^t)G^{-\gamma}(q^{t+1})K^p(q^{t+1})q^{\alpha t}\\
\leq\bigg(\frac{p+\beta}{\gamma-\theta-1}\bigg)^p
\sum_{t=\log_q a}^{\infty}\frac{k(q^{t+1})v(q^{t+1})
w(q^t)r^p(q^t)f^p(q^t)G^{\gamma(p-1)}(q^{t+1})}{g^{p-1}(q^t)
G^{p(\gamma-1)}(q^t)}q^{\alpha t},
\end{multline}
where
\begin{equation*}
G(t)=(q-1)\sum_{s=\log_q a}^{(\log_q t)-1} g(q^s)q^{\alpha s} 
\quad \text{and} \quad  
K(t)=(q-1)\sum_{s=\log_q a}^{(\log_q t)-1} r(q^s)f(q^s)q^{\alpha s}.
\end{equation*}
\end{corollary}

\begin{theorem}
\label{thm2}
Let $S_1$, $S_2$, $S_3$, $S_5$, $S_6$, $S_9$, $S_{12}$, and $S_{15}$ be satisfied. Then,
\begin{multline}
\label{eq2}
\int_{a}^{\infty}k(t)v(t)w^\sigma(t)g(t)\big(G^\sigma(t)\big)^{\alpha-\gamma+1}
F^{p-\alpha+1}(t)\Delta_\alpha t\\
\leq\Big(\frac{p+\beta-\alpha+1}{\alpha-\gamma+\theta}\Big)^p
\int_{a}^{\infty}\frac{k(t)v(t)w^\sigma(t)r^p(t)
f^p(t)\big(G^\sigma(t)\big)^{p-\gamma-\alpha+1}  
F^{1-\alpha}(t) }{g^{p-1}(t)}\Delta_\alpha t.
\end{multline}
\end{theorem}

\begin{proof}
From \eqref{parts}, we get
\begin{multline}
\label{11}
\int_{a}^{\infty}k(t)v(t)w^\sigma(t)g(t)\big(
G^\sigma(t)\big)^{\alpha-\gamma-1}F^{p-\alpha+1}(t)\Delta_{\alpha} t\\
=\Big[u(t)k(t)v(t)F^{p-\alpha+1}(t)\Big]_a^\infty
+\int_{a}^{\infty}u^\sigma(t) T_{\alpha}^{\Delta} 
\Big(-k(t)v(t)F^{p-\alpha+1}(t)\Big) \Delta_\alpha t,
\end{multline}
since
\begin{equation*}
u(t)=\int_{a}^{t}w^\sigma(s)g(s)\big(G^\sigma(s)\big)^{\alpha-\gamma-1}\Delta_\alpha s.
\end{equation*}
Applying \eqref{chain1}, \eqref{product}, and  $S_{12}$, one has
\begin{eqnarray*}
T_{\alpha}^{\Delta} \Big(w(s)G^{\alpha-\gamma}(s)\Big)
&=& T_{\alpha}^{\Delta} w(s)G^{\alpha-\gamma}(s)
+w^\sigma(s)T_{\alpha}^{\Delta} \big(G^{\alpha-\gamma}(s)\big)\\
&\geq&\theta w^\sigma(s)G^{\alpha-\gamma-1}(s)
T_{\alpha}^{\Delta} G(s)+(\alpha-\gamma)w^\sigma(s)
G^{\alpha-\gamma-1}(c) T_{\alpha}^{\Delta} G(s)
\end{eqnarray*}
for $c\in[s,\sigma(s)]$. As $T_{\alpha}^{\Delta} G(s)=g(s)\geq0$, $c\leq\sigma(s)$ 
and $0\leq \gamma <\alpha$, then we get
\begin{eqnarray*}
T_{\alpha}^{\Delta} \Big(w(s)G^{\alpha-\gamma}(s)\Big)
&\geq&\theta w^\sigma(s)g(s)\big(G^\sigma(s)\big)^{\alpha-\gamma-1}
+(\alpha-\gamma)w^\sigma(s)g(s)\big(G^\sigma(s)\big)^{\alpha-\gamma-1}\\
&=&(\alpha-\gamma+\theta)w^\sigma(s)g(s)\big(G^\sigma(s)\big)^{\alpha-\gamma-1}
\end{eqnarray*}
so that
\begin{equation*}
w^\sigma(s)g(s)\big(G^\sigma(s)\big)^{\alpha-\gamma-1}
\leq\frac{1}{\alpha-\gamma+\theta} T_{\alpha}^{\Delta} 
\Big(w(s)G^{\alpha-\gamma}(s)\Big).
\end{equation*}
Therefore,
\begin{equation}
\label{22}
\begin{split}
u^\sigma(t) =\int_{a}^{\sigma(t)}w^\sigma(s)
g(s)\big(G^\sigma(s)\big)^{\alpha-\gamma-1}\Delta_\alpha s
&\leq\frac{1}{\alpha-\gamma+\theta}\int_{a}^{\sigma(t)} 
T_{\alpha}^{\Delta} \Big(w(s)G^{\alpha-\gamma}(s)\Big)\Delta_{\alpha} s\\
&=\frac{1}{\alpha-\gamma+\theta}w^\sigma(t)\big(G^\sigma(t)\big)^{\alpha-\gamma}.
\end{split}
\end{equation}
Using \eqref{product} and \eqref{chain1}, let $c\in[t,\sigma(t)]$. Then,
\begin{eqnarray*}
T_{\alpha}^{\Delta} \Big(-k(t)v(t)F^{p-\alpha+1}(t)\Big)
&=&-\Big(\big(k(t)v(t)\big)^\Delta\big(F^\sigma(t)\big)^{p-\alpha+1}
+k(t)v(t)\big(F^{p-\alpha+1}(t)\big)^\Delta\Big)\\
&=&-\Big(k^\Delta(t)v^\sigma(t)\big(F^\sigma(t)\big)^{p-\alpha+1}
+k(t)v^\Delta(t)\big(F^\sigma(t)\big)^{p-\alpha+1}\\
&& \quad \quad +(p-\alpha+1)k(t)v(t)F^{p-\alpha}(c) T_{\alpha}^{\Delta} F(t)\Big).
\end{eqnarray*}
From $0\leq T_{\alpha}^{\Delta} k(t)$, $-r(t)f(t)=T_{\alpha}^{\Delta} F(t)\leq0$, 
$t\leq c$, $p \geq 1$ and $S_{15}$, we have
\begin{equation}
\label{44}
\begin{split}
T_{\alpha}^{\Delta} \Big(-k(t)v(t)F^{p-\alpha+1}(t)\Big)
&\leq\beta k(t)v(t)r(t)f(t)\big(F^\sigma(t)\big)^{p-\alpha}
+(p-\alpha+1)k(t)v(t)r(t)f(t)F^{p-\alpha}(t)\\
&\leq(p-\alpha+\beta+1)k(t)v(t)r(t)f(t)F^{p-\alpha}(t).
\end{split}
\end{equation}
Using \eqref{11}, \eqref{22} and \eqref{44}, 
it follows that ($F(\infty) = 0$ and $u(a) = 0$)
\begin{eqnarray*}
&&\int_{a}^{\infty}k(t)v(t)w^\sigma(t)g(t)\big(G^\sigma(t)\big)^{\alpha
-\gamma-1}\big(F^\sigma(t)\big)^{p-\alpha+1}\Delta_\alpha t\\
&&\leq\frac{(p-\alpha+\beta+1)}{\alpha-\gamma+\theta}
\int_{a}^{\infty}k(t)v(t)w^\sigma(t)r(t)
f(t)\big(G^\sigma(t)\big)^{\alpha-\gamma}F^{p-\alpha}(t)\Delta_\alpha t.
\end{eqnarray*}
Equivalently,
\begin{eqnarray*}
&&\int_{a}^{\infty}k(t)v(t)w^\sigma(t)
g(t)\big(G^\sigma(t)\big)^{\alpha-\gamma-1} F^{p-\alpha+1}(t)\Delta_\alpha t\\
&&\leq\frac{p+\beta-\alpha+1}{\alpha-\gamma+\theta}
\int_{a}^{\infty}\bigg(\big(k(t)v(t)w^\sigma(t)
g(t)\big)^{(p-1)/p}\big(G^\sigma(t)\big)^{(\alpha
-\gamma-1)(p-1)/p}F^{(p-1)(p-\alpha+1)/p}(t)\bigg)\\
&&\quad\times\bigg(\frac{\big(k(t)v(t)w^\sigma(t)\big)^{1/p}
r(t)f(t)\big(G^\sigma(t)\big)^{(p-\gamma+\alpha-1)/p} 
F^{(1-\alpha)/p}}{g^{(p-1)/p}(t)}
\bigg)\Delta_\alpha t.
\end{eqnarray*}
Applying \eqref{holder} with indexes $p$ and $p/(p-1)$, we obtain
\begin{eqnarray*}
&&\int_{a}^{\infty}k(t)v(t)w^\sigma(t)g(t)\big(G^\sigma(t)\big)^{\alpha-\gamma+1}
F^{p-\alpha+1}(t)\Delta_\alpha t\\
&&\leq\frac{p+\beta-\alpha+1}{\alpha-\gamma+\theta}
\bigg(\int_{a}^{\infty}k(t)v(t)w^\sigma(t)g(t)\big(G^\sigma(t)\big)^{\alpha-\gamma-1}
F^{p-\alpha+1}(t) \Delta_\alpha t\bigg)^{(p-1)/p}\\
&&\quad\times\bigg(\int_{a}^{\infty} \frac{k(t)v(t)w^\sigma(t)
r^p(t)f^p(t)\big(G^\sigma(t)\big)^{p-\gamma+\alpha-1} 
F^{1-\alpha}(t)}{g^{p-1}(t)} \Delta_\alpha t\bigg)^{1/p}.
\end{eqnarray*}
This gives
\begin{eqnarray*}
&&\int_{a}^{\infty}k(t)v(t)w^\sigma(t)g(t)\big(G^\sigma(t)\big)^{\alpha-\gamma+1}
F^{p-\alpha+1}(t)\Delta_\alpha t\\
&&\leq\Big(\frac{p+\beta-\alpha+1}{\alpha-\gamma+\theta}\Big)^p
\int_{a}^{\infty}\frac{k(t)v(t)w^\sigma(t)r^p(t)
f^p(t)\big(G^\sigma(t)\big)^{p-\gamma+\alpha-1}  
F^{1-\alpha}(t) }{g^{p-1}(t)}\Delta_\alpha t,
\end{eqnarray*}
which is our desired result.
\end{proof}	

\begin{corollary}
If we take $\alpha=1$ in Theorem~\ref{thm2},  
then we get the following inequality:
\begin{eqnarray*}
&&\int_{a}^{\infty}k(t)v(t)w^\sigma(t)g(t)\big(
G^\sigma(t)\big)^{-\gamma}F^p(t)\Delta t\\
&&\leq\Big(\frac{\beta+p}{1-\gamma+\theta}\Big)^p
\int_{a}^{\infty}\frac{k(t)v(t)w(t)r^p(t)f^p(t)
\big(G^\sigma(t)\big)^{p-\gamma}}{g^{p-1}(t)}\Delta t,
\end{eqnarray*}
where
\begin{equation*}
G(t)=\int_{a}^{t}g(s)\Delta s 
\quad \text{with} \quad 
G(\infty)=\infty, \quad and \quad F(t)=\int_{t}^{\infty}r(s)f(s)\Delta s,
\end{equation*}
which is Theorem~3.11 of \cite{dd1}.
\end{corollary}

\begin{remark}
Under hypotheses $S_{16}$, $S_{17}$ and  $S_{19}$,
then \eqref{eq2} of Theorem~\ref{thm2} tell us that
\begin{multline} 
\label{105} 
\int_{a}^{\infty}g(t)\big(G^\sigma(t)\big)^{\alpha-\gamma+1}
F^{p-\alpha+1}(t)\Delta_\alpha t\\
\leq\Big(\frac{p-\alpha+1}{\alpha-\gamma}\Big)^p
\int_{a}^{\infty}g(t)f^p(t)\big(G^\sigma(t)\big)^{p-\gamma+\alpha-1}  
F^{1-\alpha}(t) \Delta_\alpha t.
\end{multline}
\end{remark}

\begin{remark}
If we set $\alpha=1$ in \eqref{105},   
then we obtain inequality \eqref{saker2}.
\end{remark}

As special cases of our results, now we obtain continuous, discrete and quantum 
$\alpha$-conformable inequalities. Precisely, we consider the special cases of
time scales $\mathbb{T}=\mathbb{R}$, $\mathbb{T}=h\mathbb{Z}$, $\mathbb{T}=\mathbb{Z}$ 
and $\mathbb{T}=\overline{q^\mathbb{Z}}$.

\begin{corollary}
\label{cor5}
Putting $\mathbb{T}=\mathbb{R}$ in Theorem~\ref{thm2} 
we get from \eqref{eq2} and \eqref{r} that
\begin{multline}
\label{eqcor5}
\int_{a}^{\infty}k(t)v(t)w(t)g(t)G^{\alpha-\gamma-1}(t)
F^{p-\alpha+1}(t) t^{\alpha-1} dt\\
\leq\Big(\frac{p+\beta-\alpha+1}{\alpha-\gamma+\theta}\Big)^p
\int_{a}^{\infty} \frac{k(t)v(t)w(t)r^p(t)f^p(t)G^{p-\gamma
+\alpha-1}(t)F^{1-\alpha}(t)}{g^{p-1}(t)} t^{\alpha-1} dt
\end{multline}
with
\begin{equation*}
G(t)=\int_{a}^{t}g(s) s^{\alpha-1} ds 
\quad \text{and} \quad  
F(t)=\int_{t}^{\infty}r(s)f(s) s^{\alpha-1} ds.
\end{equation*}
\end{corollary}

\begin{remark}
Under assumptions $S_{16}$, $S_{17}$, $S_{19}$ and  $S_{20}$,
inequality \eqref{eqcor5} of Corollary~\ref{cor5} asserts that
\begin{multline}
\label{eqcor55}
\int_{0}^{\infty}g(t)G^{\alpha-\gamma-1}(t)
F^{p-\alpha+1}(t) t^{\alpha-1} dt\\
\leq\Big(\frac{p+\beta-\alpha+1}{\alpha-\gamma+\theta}\Big)^p
\int_{0}^{\infty}g(t)f^p(t) G^{p-\gamma+\alpha-1}(t)
F^{1-\alpha}(t) t^{\alpha-1} dt.
\end{multline}
\end{remark}

\begin{remark}
If $\alpha=1$, then inequality \eqref{eqcor55} reduces to \eqref{h25}.
\end{remark}

\begin{remark}
With hypotheses $S_{16}$, $S_{18}$, $S_{19}$ and $S_{20}$,
inequality \eqref{eqcor5} of Corollary~\ref{cor5} gives us that
\begin{multline}
\label{eqcor66}
\int_{0}^{\infty}G^{\alpha-\gamma-1}(t) \bigg(
\int_{0}^{\infty}f(s)s^{\alpha-1}\bigg)^{p-\alpha+1} t^{\alpha-1} dt\\
\leq\Big(\frac{p-\alpha+1}{\alpha-\gamma}\Big)^p
\int_{a}^{\infty}f^p(t)G^{p-\gamma+\alpha-1}(t)\bigg(
\int_{0}^{\infty}f(s)s^{\alpha-1}\bigg)^{1-\alpha} t^{\alpha-1} dt.
\end{multline}
\end{remark}

\begin{remark}
If we set $\alpha=1$, 
then the inequality \eqref{eqcor66} simplifies to \eqref{a6}.
\end{remark}

\begin{remark}
With $S_{16}$, $S_{18}$, $S_{19}$, $S_{20}$, and $S_{22}$, 
\eqref{eqcor5} of Corollary~\ref{cor5} gives inequality
\begin{multline}
\label{eqcor99f}
\int_{0}^{\infty}G^{\alpha-p-1}(t)\bigg(
\int_{0}^{\infty}f(s)s^{\alpha-1}\bigg)^{p-\alpha+1} t^{\alpha-1} dt\\
\leq\Big(\frac{p-\alpha+1}{\alpha-p}\Big)^p\int_{a}^{\infty}f^p(t)
G^{\alpha-1}(t)\bigg(\int_{0}^{\infty}f(s)s^{\alpha-1}\bigg)^{1-\alpha} t^{\alpha-1} dt.
\end{multline}
\end{remark}

\begin{remark}
In the particular case $\alpha=1$, inequality \eqref{eqcor99f} gives us \eqref{a11}.
\end{remark}

\begin{corollary}
\label{cor6}
Choosing $\mathbb{T}=h\mathbb{Z}$ in Theorem~\ref{thm2}, 
we obtain from inequality \eqref{eq2} that
\begin{multline}
\sum_{t=\frac{a}{h}}^{\infty}k(ht)v(ht)w(ht+h)g(ht)
G^{-\gamma}(ht+h)F^p(ht)t^{\alpha-1}\\
\leq\Big(\frac{p+\beta}{1-\gamma+\theta}\Big)^p
\sum_{t=\frac{a}{h}}^{\infty}\frac{k(ht)v(ht)w(ht+h)r^p(ht)f^p(ht)
G^{p-\gamma}(ht+h)}{g^{p-1}(ht)}t^{\alpha-1},
\end{multline}
where
\begin{equation*}
G(t)=h^{\alpha}\sum_{s=\frac{a}{h}}^{\frac{t}{h}-1}g(hs)s^{\alpha-1} 
\quad \text{and} \quad  
F(t)=h^{\alpha}\sum_{s=\frac{t}{h}}^{\infty}r(hs)f(hs)s^{\alpha-1}.
\end{equation*}
\end{corollary}

\begin{corollary}
\label{cor7}
Putting $h=1$ in Corollary~\ref{cor6}, 
that is, for the discrete time-scale $\mathbb{T}=\mathbb{Z}$,
we obtain from \eqref{eq2} the inequality
\begin{multline}
\label{eqcor7}
\sum_{t=a}^{\infty}k(t)v(t)w(t+1)g(t)G^{-\gamma}(t+1)F^p(t)t^{\alpha-1}\\
\leq\Big(\frac{p+\beta}{1-\gamma+\theta}\Big)^p\sum_{t=a}^{\infty}
\frac{k(t)v(t)w(t+1)r^p(t)f^p(t) G^{p-\gamma}(t+1)}{g^{p-1}(t)}t^{\alpha-1},
\end{multline}
where
\begin{equation*}
G(t)=\sum_{s=a}^{t-1}g(s)s^{\alpha-1} 
\quad \text{and} \quad  
F(t)=\sum_{s=t}^{\infty}r(s)f(s)ss^{\alpha-1}.
\end{equation*}
\end{corollary}

\begin{remark}
With $S_{16}$, $S_{17}$, $S_{19}$ and $S_{21}$, 
then \eqref{eqcor7} of Corollary~\ref{cor7} reduces to \eqref{h21}.
\end{remark}

\begin{corollary}
\label{cor8:b}
Putting $\mathbb{T}=\overline{q^\mathbb{Z}}$ 
in Theorem~\ref{thm2}, it follows from
\eqref{q^z} that \eqref{eq1} simplifies to
\begin{multline}
\sum_{t=\log_q a}^{\infty}k(q^t)v(q^t)w(q^{t+1})g(q^t)
G^{-\gamma}(q^{t+1})F^p(q^t)q^{\alpha t}\\
\leq\Big(\frac{p+\beta}{1-\gamma+\theta}\Big)^p
\sum_{t=\log_q a}^{\infty}\frac{k(q^t)v(q^t)w(q^{t+1})r^p(q^t)f^p(q^t)
G^{p-\gamma}(q^{t+1})}{g^{p-1}(q^t)}q^{\alpha t}
\end{multline}
with
\begin{equation*}
G(t)=(q-1)\sum_{s=\log_q a}^{(\log_q t)-1}g(q^s)q^{\alpha s} 
\quad \text{and} \quad  
F(t)=(q-1)\sum_{s=\log_q t}^{\infty} r(q^s)f(q^s)q^{\alpha s}.
\end{equation*}
\end{corollary}

\begin{theorem}
\label{thm3}
Let $S_1$, $S_2$, $S_3$, $S_5$, $S_7$, $S_8$, $S_{11}$, $S_{15}$ be satisfied. Then,
\begin{multline}
\label{eq3}
\int_{a}^{\infty}k^\sigma(t)v^\sigma(t)w(t)g(t)H^{\alpha-\gamma-1}(t)\big(
K^\sigma(t)\big)^{p-\alpha+1}\Delta_\alpha t\\
\leq\Big(\frac{p-\alpha+\beta+1}{\alpha-\gamma+\theta}\Big)^p
\int_{a}^{\infty}\frac{k^\sigma(t)v^\sigma(t)w(t)r^p(t)f^p(t)
H^{p-\gamma+\alpha-1}(t)\big(K^\sigma\big(t))^{(1-\alpha)}}{g^{p-1}(t)}\Delta_\alpha t.
\end{multline}
\end{theorem}

\begin{proof}
Using the $\alpha$-conformable integration by parts formula 
on time scales  \eqref{parts} with
\begin{equation*}
T_\alpha^\Delta u(t)=w(t)g(t)H^{\alpha-\gamma-1}(t) 
\qquad \text{and} \qquad z^\sigma(t)=k^\sigma(t)
v^\sigma(t)\big(K^\sigma(t)\big)^{p-\alpha+1},
\end{equation*}
we have
\begin{multline}
\label{111}
\int_{a}^{\infty}k^\sigma(t)v^\sigma(t)w(t)g(t)
H^{\alpha-\gamma-1}(t)\big(K^\sigma(t)\big)^{p-\alpha+1}\Delta_\alpha t\\
=\Big[u(t)k(t)v(t)K^{p-\alpha+1}(t)\Big]_a^\infty
+\int_{a}^{\infty}\big(-u(t)\big) T_\alpha^\Delta \Big(k(t)v(t)
K^{p-\alpha+1}(t)\Big) \Delta_\alpha t,
\end{multline}
where
\begin{equation*}
u(t)=-\int_{t}^{\infty}w(s)g(s)H^{\alpha-\gamma-1}(s)\Delta_\alpha s.
\end{equation*}
From  \eqref{chain1}, \eqref{product}, and $S_{11}$, then
\begin{eqnarray*}
T_\alpha^\Delta \Big(-w(s)H^{\alpha-\gamma}(s)\Big)
&=&-\Big( T_\alpha^\Delta w(s)\big(H^\sigma(s)\big)^{\alpha-\gamma}
+w(s) T_\alpha^\Delta \big(H^{\alpha-\gamma}(s)\big)\Big)\\
&\geq&-\Big(\theta w(s)\big(H^\sigma(s)\big)^{\alpha-\gamma-1} 
T_\alpha^\Delta  H(s)+(\alpha-\gamma)w(s)H^{\alpha-\gamma-1}(c) T_\alpha^\Delta H(s)\Big).
\end{eqnarray*}
Since $s\leq c$ , $0\leq T_\alpha^\Delta H(s)=-g(s)$ and $\alpha>\gamma>0$, we get
\begin{eqnarray*}
T_\alpha^\Delta  \Big(-w(s)H^{\alpha-\gamma}(s)\Big)
&\geq&\theta w(s)g(s)H^{\alpha-\gamma-1}(s)
+(\alpha-\gamma)w(s)g(s)H^{\alpha-\gamma-1}(s)\\
&=&(\alpha-\gamma+\theta)w(s)g(s)H^{\alpha-\gamma-1}(s).
\end{eqnarray*}
Thus,
\begin{equation*}
w(s)g(s)H^{\alpha-\gamma-1}(s)\leq\frac{1}{\alpha-\gamma+\theta} 
T_\alpha^\Delta \Big(-w(s)H^{\alpha-\gamma}(s)\Big).
\end{equation*}
Hence,
\begin{equation}
\label{222}
\begin{split}
-u(t)=\int_{t}^{\infty}w(s)g(s)H^{\alpha-\gamma-1}(s)\Delta_\alpha s
&\leq \frac{1}{\alpha-\gamma+\theta}\int_{t}^{\infty} T_\alpha^\Delta  
\Big(-w(s)H^{\alpha-\gamma}(s)\Big)\Delta_\alpha s\\
&=\frac{1}{\alpha-\gamma+\theta}w(t)H^{\alpha-\gamma}(t).
\end{split}
\end{equation}
Using \eqref{product} and \eqref{chain1}, one obtains
\begin{multline*}
T_\alpha^\Delta \Big(k(t)v(t)K^{p-\alpha+1}(t)\Big)
= T_\alpha^\Delta \big(k(t)v(t)\big) K^{p-\alpha+1}(t)
+k^\sigma(t)v^\sigma(t)  T_\alpha^\Delta \big(K^{p-\alpha+1}(t)\big)\\
= T_\alpha^\Delta k(t)v(t)K^{p-\alpha+1}(t)+k^\sigma(t) T_\alpha^\Delta
v(t)K^{p-\alpha+1}(t)+(p-\alpha+1)k^\sigma(t)v^\sigma(t)K^{p-\alpha}(c)  
T_\alpha^\Delta K(t)
\end{multline*}
with $c\in[t,\sigma(t)]$. Considering $r(t)f(t)=T_\alpha^\Delta K(t)\geq0$, 
$0\leq T_\alpha^\Delta k(t)$,  $c\leq\sigma(t)$, $1\leq p\geq$ and $S_{14}$, 
we arrive to
\begin{equation}
\label{444}
\begin{split}
T_\alpha^\Delta \Big(k(t)v(t)K^p(t)\Big)
&\leq\beta k^\sigma(t)v^\sigma(t)r(t)f(t)K^{p-\alpha}(t)
+(p-\alpha+1)k^\sigma(t)v^\sigma(t)r(t)f(t)\big(
K^\sigma(t)\big)^{p-\alpha}\\
&\leq(p-\alpha+\beta+1)k^\sigma(t)v^\sigma(t)r(t)f(t)\big(
K^\sigma(t)\big)^{p-\alpha}.
\end{split}
\end{equation}
Now, combining \eqref{111}, \eqref{222} and \eqref{444}, 
we get ($K(a)=0$ and $u(\infty)=0$)
\begin{multline}
\label{cann}
\int_{a}^{\infty}k^\sigma(t)v^\sigma(t)w(t)g(t)
H^{\alpha-\gamma-1}(t)\big(K^\sigma(t)\big)^{p-\alpha+1}\Delta_\alpha t\\
\leq\frac{p-\alpha+\beta+1}{\alpha-\gamma+\theta}
\int_{a}^{\infty}k^\sigma(t)v^\sigma(t)w(t)r(t)f(t)
H^{\alpha-\gamma}(t)\big(K^\sigma(t)\big)^{p-\alpha}\Delta_\alpha t.
\end{multline}
Inequality \eqref{cann} becomes
\begin{eqnarray*}
&&\int_{a}^{\infty}k^\sigma(t)v^\sigma(t)w(t)g(t)
H^{\alpha-\gamma+1}(t)\big(K^\sigma(t)\big)^{p-\alpha+1}\Delta_\alpha t\\
&&\leq\frac{p-\alpha+\beta+1}{\alpha-\gamma+\theta}
\int_{a}^{\infty}\bigg(\big(k^\sigma(t)v^\sigma(t)w(t)g(t)\big)^{(p-1)/p}
H^{(\alpha-\gamma-1)(p-1)/p}(t)\big(K^\sigma\big(t))^{(p-1)(p-\alpha+1)/p}\bigg)\\
&&\quad\times\bigg(\frac{\big(k^\sigma(t)v^\sigma(t)w(t)\big)^{1/p}r(t)
f(t)H^{(p-\gamma+\alpha-1)/p}(t)\big(K^\sigma\big(t))^{(1-\alpha)/p} }{g^{(p-1)/p}(t)}
\bigg)\Delta_\alpha t.
\end{eqnarray*}
Using the The $\alpha$-conformable H\"{o}lder inequality \eqref{holder} 
with $p$ and $p/(p-1)$ indices, we obtain that
\begin{eqnarray*}
&&\int_{a}^{\infty}k^\sigma(t)v^\sigma(t)w(t)g(t)
H^{\alpha-\gamma-1}(t)\big(K^\sigma(t)\big)^{p-\alpha+1}\Delta_\alpha t\\
&&\leq\frac{p-\alpha+\beta+1}{\alpha-\gamma+\theta}
\bigg(\int_{a}^{\infty}k^\sigma(t)v^\sigma(t)w(t)g(t)
H^{\alpha-\gamma-1}(t)\big(K^\sigma\big(t)\big)^{p-\alpha+1}\Delta_\alpha t\bigg)^{(p-1)/p}\\
&&\quad\times\bigg(\int_{a}^{\infty}\frac{k^\sigma(t)v^\sigma(t)w(t)r^p(t)
f^p(t)H^{p-\gamma+\alpha-1}(t)\big(K^\sigma\big(t))^{(1-\alpha)}}{g^{p-1}(t)}
\Delta_\alpha t\bigg)^{1/p}.
\end{eqnarray*}
This implies that
\begin{eqnarray*}
&&\int_{a}^{\infty}k^\sigma(t)v^\sigma(t)w(t)g(t)
H^{\alpha-\gamma-1}(t)\big(K^\sigma(t)\big)^{p-\alpha+1}\Delta_\alpha t\\
&&\leq\Big(\frac{p-\alpha+\beta+1}{\alpha-\gamma+\theta}\Big)^p
\int_{a}^{\infty}\frac{k^\sigma(t)v^\sigma(t)w(t)r^p(t)f^p(t)
H^{p-\gamma+\alpha-1}(t)\big(K^\sigma\big(t))^{(1-\alpha)}}{g^{p-1}(t)}\Delta_\alpha t,
\end{eqnarray*}
which is the desired result. 
\end{proof}	

\begin{corollary}
Putting $\alpha=1$ in Theorem~\ref{thm3}, then inequality \eqref{eq3} reduces to
\begin{eqnarray*}
&&\int_{a}^{\infty}k^\sigma(t)v^\sigma(t)w(t)
g(t)H^{-\gamma}(t)\big(K^\sigma(t)\big)^p\Delta t\\
&&\leq\Big(\frac{p+\beta}{1-\gamma+\theta}\Big)^p
\int_{a}^{\infty}\frac{k^\sigma(t)v^\sigma(t)
w(t)r^p(t)f^p(t)H^{p-\gamma}(t)}{g^{p-1}(t)}\Delta t
\end{eqnarray*}
where
\begin{equation*}
H(t)=\int_{t}^{\infty}g(s)\Delta s 
\quad \text{and} \quad 
K(t)=\int_{a}^{t}r(s)f(s)\Delta s, \quad t\in[a,\infty)_\mathbb{T},
\end{equation*}
which is Theorem~3.21 of \cite{dd1}.
\end{corollary}

\begin{remark}
Under $S_{16}$, $S_{17}$ and $S_{19}$,
\eqref{eq3} of Theorem~\ref{thm3} tell us that
\begin{multline} 
\label{306}
\int_{a}^{\infty} g(t)H^{\alpha-\gamma-1}(t)\big(
K^\sigma(t)\big)^{p-\alpha+1}\Delta_\alpha t\\
\leq\Big(\frac{p-\alpha+1}{\alpha-\gamma}\Big)^p\int_{a}^{\infty}
g(t)f^p(t)H^{p-\gamma+\alpha-1}(t)\big(
K^\sigma\big(t))^{(1-\alpha)}\Delta_\alpha t.
\end{multline}
\end{remark}

\begin{remark}
For $\alpha=1$, inequality \eqref{306} gives us \eqref{saker3}.
\end{remark}

Now, as special cases of our results, we obtain continuous, 
discrete and quantum $\alpha$-conformable inequalities. 
For that, we fix the time scale as $\mathbb{T}=\mathbb{R}$, 
$\mathbb{T}=h\mathbb{Z}$, $\mathbb{T}=\mathbb{Z}$, 
or $\mathbb{T}=\overline{q^\mathbb{Z}}$.

\begin{corollary}
\label{cor9}
Putting  $\mathbb{T}=\mathbb{R}$ in Theorem~\ref{thm3},  
it follows from \eqref{r} and \eqref{eq3} that
\begin{multline}
\label{eqcor9}
\int_{a}^{\infty}k(t)v(t)w(t)g(t)
H^{\alpha-\gamma-1}(t)K^{p-\alpha+1}(t)t^{\alpha-1} dt\\
\leq\Big(\frac{p+\beta-\alpha+1}{\alpha-\gamma+\theta}\Big)^p
\int_{a}^{\infty}\frac{k(t)v(t)w(t)r^p(t)
f^p(t)H^{p-\gamma+\alpha-1}(t)K^{1-\alpha}(t)}{g^{p-1}(t)}t^{\alpha-1} dt,
\end{multline}
where
\begin{equation*}
H(t)=\int_{t}^{\infty}g(s) s^{\alpha-1}ds 
\quad \text{and} \quad  
K(t)=\int_{a}^{t}r(s)f(s)s^{\alpha-1} ds.
\end{equation*}
\end{corollary}

\begin{remark}
Under $S_{16}$, $S_{17}$, $S_{19}$, and $S_{20}$,
inequality \eqref{eqcor9} of Corollary~\ref{cor9} gives us
\begin{multline}
\label{eqcor99}
\int_{0}^{\infty}g(t)H^{\alpha-\gamma-1}(t)K^{p-\alpha+1}(t)t^{\alpha-1} dt\\
\leq\Big(\frac{p-\alpha+1}{\alpha-\gamma}\Big)^p
\int_{0}^{\infty}g(t)f^p(t)H^{p-\gamma+\alpha-1}(t)K^{1-\alpha}(t)t^{\alpha-1} dt.
\end{multline}
\end{remark}

\begin{remark}
If we take $\alpha=1$, then inequality \eqref{eqcor99} reduces to
\begin{equation}
\int_{0}^{\infty}\frac{r(t)\Big(\int_{0}^{t}g(s)f(s)ds\Big)^p}{\Big(
\int_{t}^{\infty}g(s)ds\Big)^\gamma}dt
\leq\Big(\frac{p}{1-\gamma}\Big)^p
\int_{0}^{\infty}g(t)f^p(t)\Big(\int_{t}^{\infty}g(s)ds\Big)^{p-\gamma}dt,
\end{equation}
that is, we get the continuous analog of \eqref{h30}.
\end{remark}

\begin{corollary}
\label{cor10}
Putting $\mathbb{T}=h\mathbb{Z}$ in Theorem \ref{thm3}, then inequality \eqref{eq3} gives
\begin{multline}
\sum_{t=\frac{a}{h}}^{\infty}k(ht+h)v(ht+h)w(ht)g(ht)H^{-\gamma}(ht)K^p(ht+h)t^{\alpha-1}\\
\leq\Big(\frac{p+\beta}{1-\gamma+\theta}\Big)^p\sum_{t=\frac{a}{h}}^{\infty}\frac{k(ht+h)v(ht+h)
w(ht)r^p(ht)f^p(ht)H^{p-\gamma}(ht)}
{g^{p-1}(ht)}t^{\alpha-1},
\end{multline}
where
\begin{equation*}
H(t)=h^{\alpha}\sum_{s=\frac{t}{h}}^{\infty}g(hs)s^{\alpha-1} 
\quad \text{and} \quad  
K(t)=h^{\alpha}\sum_{s=\frac{a}{h}}^{\frac{t}{h}-1}r(hs)f(hs)s^{\alpha-1}.
\end{equation*}
\end{corollary}

\begin{corollary}
\label{cor11}
Putting $h=1$ in Corollary~\ref{cor10}, that is, 
fixing the time scale to be $\mathbb{T}=\mathbb{Z}$, 
then one obtains that
\begin{multline}
\label{eqcor11}
\sum_{t=a}^{\infty}k(t+1)v(t+1)w(t)g(t)H^{-\gamma}(t)K^p(t+1)t^{\alpha-1}\\
\leq\Big(\frac{p+\beta}{1-\alpha+\theta}\Big)^p
\sum_{t=a}^{\infty}\frac{k(t+1)v(t+1)w(t)r^p(t)f^p(t)H^{p-\gamma}(t)}
{g^{p-1}(t)}t^{\alpha-1},
\end{multline}
where
\begin{equation*}
H(t)=\sum_{s=t}^{\infty}g(s)s^{\alpha-1} 
\quad \text{and} \quad  
K(t)=\sum_{s=a}^{t-1}r(s)f(s)s^{\alpha-1}.
\end{equation*}
\end{corollary}

\begin{remark}
Under assumptions $S_{16}$, $S_{17}$, $S_{19}$ and $S_{21}$, 
inequality \eqref{eqcor11} of Corollary~\ref{cor11} gives us 
\eqref{h30}.
\end{remark}

\begin{corollary}
\label{cor12}
Choosing $\mathbb{T}=\overline{q^\mathbb{Z}}$ in Theorem~\ref{thm3}, 
it follows from \eqref{q^z} and \eqref{eq3} that
\begin{multline}
\sum_{t=\log_q a}^{\infty}k(q^{t+1})v(q^{t+1})
w(q^t)g(q^t)H^{-\gamma}(q^t)K^p(q^{t+1})q^{\alpha t}\\
\leq\Big(\frac{p+\beta}{1-\gamma+\theta}\Big)^p\sum_{t=\log_q a}^{\infty}
\frac{k(q^{t+1})v(q^{t+1})w(q^t)r^p(q^t)f^p(q^t)
H^{p-\gamma}(q^t)}{g^{p-1}(q^t)}q^{\alpha t},
\end{multline}
where
\begin{equation*}
H(t)=(q-1)\sum_{s=\log_q t}^{\infty} g(q^s)q^{\alpha s} 
\quad \text{and} \quad  
K(t)=(q-1)\sum_{s=\log_q a}^{(\log_q t)-1} r(q^s)f(q^s)q^{\alpha s}.
\end{equation*}
\end{corollary}

\begin{theorem}
\label{thm4}
Let $S_1$, $S_2$, $S_3$, $S_4$, $S_7$, $S_9$, $S_{13}$, 
and $S_{15}$ be satisfied. Then,
\begin{multline}
\label{eq4}
\int_{a}^{\infty}k(t)v(t)w^\sigma(t)g(t)
H^{\alpha-\gamma-1}(t)F^{p-\alpha+1}(t)\Delta_\alpha t\\ 
\leq\Big(\frac{p+\beta-\alpha+1}{\gamma-\theta-\alpha}\Big)^p
\int_{a}^{\infty}\frac{k(t)v(t)w^\sigma(t)r^p(t)f^p(t)H^{(1-\alpha+\gamma)(p-1)}(t)
 F^{(1-\alpha)/p}(t)}{g^{p-1}(t)\big(H^\sigma(t)\big)^{p(\gamma-\alpha)}}
\Delta_\alpha t.
\end{multline}
\end{theorem}

\begin{proof}
Using \eqref{parts}, we get
\begin{multline}
\label{1111}
\int_{a}^{\infty}k(t)v(t)w^\sigma(t)g(t)
H^{\alpha-\gamma-1}(t)F^{p-\alpha+1}(t)\Delta_\alpha t\\
=\Big[u(t)k(t)v(t)F^{p-\alpha+1}(t)\Big]_a^\infty
+\int_{a}^{\infty}u^\sigma(t)T_{\alpha}^{\Delta} 
\Big(-k(t)v(t)F^{p-\alpha+1}(t)\Big) \Delta_\alpha t
\end{multline}
with
\begin{equation*}
u(t)=\int_{a}^{t}w^\sigma(s)g(s)
H^{\alpha-\gamma-1}(s)\Delta_\alpha s.
\end{equation*}
Now, using \eqref{chain1}, \eqref{product}, and $S_{13}$, it follows that
\begin{eqnarray*}
T_\alpha^{\Delta} \Big(w(s)H^{\alpha-\gamma}(s)\Big)
&=&w^\Delta(s)H^{\alpha-\gamma}(s)+w^\sigma(s) 
T_\alpha^{\Delta} \big(H^{\alpha-\gamma}(s)\big)\\
&\geq&\theta w^\sigma(s) T_\alpha^{\Delta} 
H(s)H^{\alpha-\gamma-1}(s)+(\alpha-\gamma)w^\sigma(s)
H^{\alpha-\gamma-1}(c) T_\alpha^{\Delta} H(s).
\end{eqnarray*}
Since $T_\alpha^{\Delta} H(s)=-g(s)\leq0$, 
$c\geq s$ and $\gamma>1$, we obtain that
\begin{eqnarray*}
T_\alpha^{\Delta} \Big(w(s)H^{\alpha-\gamma}(s)\Big)
&\geq&-\theta w^\sigma(s)g(s)H^{\alpha-\gamma-1}(s)
+(\gamma-\alpha)w^\sigma(s)g(s)H^{\alpha-\gamma-1}(s)\\
&=&(\gamma-\theta-\alpha)w^\sigma(s)g(s)H^{\alpha-\gamma-1}(s).
\end{eqnarray*}
This gives us that
\begin{equation*}
w^\sigma(s)g(s)H^{\alpha-\gamma-1}(s)
\leq\frac{1}{\gamma-\theta-\alpha} 
T_\alpha^{\Delta} \Big(w(s)H^{\alpha-\gamma}(s)\Big).
\end{equation*}
Therefore,
\begin{equation}
\label{2222}
\begin{split}
u^\sigma(t)=\int_{a}^{\sigma(t)}w^\sigma(s)g(s)H^{\alpha-\gamma-1}(s)\Delta_\alpha s
&\leq\frac{1}{\gamma-\theta-\alpha}\int_{a}^{\sigma(t)} 
T_\alpha^{\Delta} \Big(w(s)H^{\alpha-\gamma}(s)\Big)\Delta_\alpha s\\
&=\frac{1}{\gamma-\theta-\alpha}\Big(w^\sigma(t)\big(
H^\sigma(t)\big)^{\alpha-\gamma}-w(a)H^{\alpha-\gamma}(a)\Big)\\
&\leq\frac{1}{\gamma-\theta-\alpha}w^\sigma(t)\big(H^\sigma(t)\big)^{\alpha-\gamma}.
\end{split}
\end{equation}
Let $c\in[t,\sigma(t)]$. Then, using \eqref{product} and \eqref{chain1}, 
one has
\begin{eqnarray*}
T_\alpha^{\Delta} \Big(-k(t)v(t)F^p(t)\Big)
&=&- \Big( T_\alpha^{\Delta} \big(k(t)v(t)\big)\big(
F^\sigma(t)\big)^{p-\alpha+1}+k(t)v(t) 
T_\alpha^{\Delta} \big(F^{p-\alpha+1}(t)\big)\Big)\\
&=&-\Big( T_\alpha^{\Delta} k(t)v^\sigma(t)\big(
F^\sigma(t)\big)^{p-\alpha+1}+k(t) T_\alpha^{\Delta} 
v(t)\big(F^\sigma(t)\big)^{p-\alpha+1}\\
&& \quad \quad+(p-\alpha+1)k(t)v(t)
F^{p-\alpha}(c) T_\alpha^{\Delta}F(t)\Big).
\end{eqnarray*}
Since  $t\leq c$, $p > 1$, $0\leq T_\alpha^{\Delta} k(t)0$, 
$-r(t)f(t)=T_\alpha^{\Delta} F(t)\leq0$, and $S_{15}$, we get
\begin{equation}
\label{4444}
\begin{split}
T_\alpha^{\Delta} \Big(-k(t)v(t)F^{p-\alpha+1}(t)\Big)
&\leq\beta k(t)v(t)r(t)f(t)\big(F^\sigma(t)\big)^{p-\alpha}
+(p-\alpha+1)k(t)v(t)r(t)f(t)F^{p-\alpha}(t)\\
&\leq(p+\beta+\alpha-1)k(t)v(t)r(t)f(t)F^{p-\alpha}(t).
\end{split}
\end{equation}
From \eqref{1111}, \eqref{2222} and \eqref{4444}, 
we obtain that ($u(a) = 0$ and $F(\infty) = 0$)
\begin{eqnarray*}
&&\int_{a}^{\infty}k(t)v(t)w^\sigma(t)g(t)
H^{\alpha-\gamma-1}(t)F^{p-\alpha+1}(t)\Delta_\alpha t\\
&&\leq\frac{p+\beta+\alpha-1}{\gamma-\theta-\alpha}
\int_{a}^{\infty}k(t)v(t)w^\sigma(t)r(t)
f(t)\big(H^\sigma(t)\big)^{\alpha-\gamma}F^{p-\alpha}(t)\Delta_\alpha t,
\end{eqnarray*}
or, equivalently,
\begin{eqnarray*}
&&\int_{a}^{\infty}k(t)v(t)w^\sigma(t)g(t)
H^{\alpha-\gamma-1}(t)F^{p-\alpha+1}(t)\Delta_\alpha t\\
&&\leq\frac{p+\beta+\alpha-1}{\gamma-\theta-\alpha}
\int_{a}^{\infty}\bigg(\big(k(t)v(t)w^\sigma(t)g(t)\big)^{(p-1)/p}
H^{(\alpha-\gamma-1)(p-1)/p}(t)F^{(p-\alpha+1)(p-1)/p}(t)\bigg)\\
&&\quad\times\bigg(\frac{\big(k(t)v(t)w^\sigma(t)\big)^{1/p}
r(t)f(t)H^{(1-\alpha+\gamma)(p-1)/p}(t)F^{(1-\alpha)/p}(t)}{g^{(p-1)/p}(t)
\big(H^\sigma(t)\big)^{\gamma-\alpha}}\bigg)\Delta_\alpha t.
\end{eqnarray*}
Applying the dynamic H\"{o}lder inequality \eqref{holder} 
with indices $p$ and $p/(p-1)$, we get
\begin{eqnarray*}
&&\int_{a}^{\infty}k(t)v(t)w^\sigma(t)g(t)
H^{\alpha-\gamma-1}(t)F^{p-\alpha+1}(t)\Delta_\alpha t\\
&&\leq\frac{p+\beta-\alpha+1}{\gamma-\theta-\alpha}
\bigg(\int_{a}^{\infty}k(t)v(t)w^\sigma(t)g(t)
H^{\alpha-\gamma-1}(t)F^{p-\alpha+1}(t)\Delta_\alpha t\bigg)^{(p-1)/p}\\
&&\quad\times\bigg(\int_{a}^{\infty}\frac{k(t)v(t)
w^\sigma(t)r^p(t)f^p(t)H^{(1-\alpha+\gamma)(p-1)}(t)
F^{(1-\alpha)}(t)}{g^{p-1}(t)\big(H^\sigma(t)\big)^{p(\gamma-\alpha)}}
\Delta_\alpha t\bigg)^{1/p},
\end{eqnarray*}
which implies that
\begin{eqnarray*}
&&\int_{a}^{\infty}k(t)v(t)w^\sigma(t)g(t)
H^{\alpha-\gamma-1}(t)F^{p-\alpha+1}(t)\Delta_\alpha t\\
&&\leq\Big(\frac{p+\beta-\alpha+1}{\gamma-\theta-\alpha}\Big)^p
\int_{a}^{\infty}\frac{k(t)v(t)w^\sigma(t)
r^p(t)f^p(t)H^{(1-\alpha+\gamma)(p-1)}(t) 
F^{(1-\alpha)}(t)}{g^{p-1}(t)\big(H^\sigma(t)\big)^{p(\gamma-\alpha)}}
\Delta_\alpha t.
\end{eqnarray*}
The proof is complete.
\end{proof}	

\begin{corollary}
If one takes $\alpha=1$ in Theorem~\ref{thm4},  
then inequality \eqref{eq3} reduces to
\begin{multline}
\label{eq44}
\int_{a}^{\infty}v(t)k(t)w^\sigma(t)g(t)H^{-\gamma}(t)F^p(t)\Delta t\\
\leq\Big(\frac{p+\beta}{\gamma-\theta-1}\Big)^p\int_{a}^{\infty}
\frac{k(t)v(t)w^\sigma(t)r^p(t)f^p(t)
H^{\gamma(p-1)}(t)}{g^{p-1}(t)\big(H^\sigma(t)\big)^{p(\gamma-1)}}\Delta t
\end{multline}
with
\begin{equation*}
H(t)=\int_{t}^{\infty}g(s)\Delta s 
\quad and \quad  
F(t)=\int_{t}^{\infty}r(s)f(s)\Delta s, \quad t\in[a,\infty)_\mathbb{T},
\end{equation*}
which is Theorem~3.29 of \cite{dd1}.
\end{corollary}

\begin{remark}
Under assumptions $S_{16}$, $S_{17}$, and $S_{19}$,
\eqref{eq4} of Theorem~\ref{thm4} gives
\begin{multline} 
\label{556}
\int_{a}^{\infty}g(t)H^{\alpha-\gamma-1}(t)F^{p-\alpha+1}(t)\Delta_\alpha t\\
\leq\Big(\frac{p-\alpha+1}{\gamma-\alpha}\Big)^p\int_{a}^{\infty}
\frac{g(t)f^p(t)H^{(1-\alpha+\gamma)(p-1)}(t) 
F^{(1-\alpha)}(t)}{\big(H^\sigma(t)\big)^{p(\gamma-\alpha)}}\Delta_\alpha t.
\end{multline}
\end{remark}

\begin{remark}
If we take $\alpha=1$ in inequality \eqref{556},  
then we obtain \eqref{saker4}.
\end{remark}

Now, as special cases of our results, we give continuous, discrete, and quantum 
$\alpha$-conformable inequalities. Namely, the following results are obtained 
by choosing $\mathbb{T}$ as time scales $\mathbb{T}=\mathbb{R}$, $\mathbb{T}=h\mathbb{Z}$, 
$\mathbb{T}=\mathbb{Z}$ and $\mathbb{T}=\overline{q^\mathbb{Z}}$.

\begin{corollary}
\label{cor13}
Putting  $\mathbb{T}=\mathbb{R}$ in Theorem \ref{thm4}, 
then inequality \eqref{eq4} becomes
\begin{multline}
\label{eqcor13f}
\int_{a}^{\infty}k(t)v(t)w(t)g(t)H^{\alpha-\gamma-1}(t)F^{p-\alpha+1}(t) t^{\alpha-1} dt\\
\leq\Big(\frac{p+\beta-\alpha+1}{\gamma-\theta-\alpha}\Big)^p
\int_{a}^{\infty}\frac{k(t)v(t)w(t)r^p(t)f^p(t)H^{p-\gamma+\alpha-1}(t) 
F^{1-\alpha}(t)}{g^{p-1}(t)} t^{\alpha-1} dt,
\end{multline}
where
\begin{equation*}
H(t)=\int_{t}^{\infty}g(s)s^{\alpha-1} ds 
\quad \text{and} \quad  
F(t)=\int_{t}^{\infty}r(s)f(s) s^{\alpha-1} ds.
\end{equation*}
\end{corollary}

\begin{remark}
Under hypotheses $S_{16}$, $S_{17}$, $S_{19}$ and $S_{20}$,
inequality \eqref{eqcor13f} of Corollary~\ref{cor13} gives
\begin{multline}
\label{eqcor133}
\int_{0}^{\infty}g(t)H^{\alpha-\gamma-1}(t)
F^{p-\alpha+1}(t) t^{\alpha-1} dt\\
\leq\Big(\frac{p-\alpha+1}{\gamma-\alpha}\Big)^p
\int_{0}^{\infty}g(t)f^p(t)H^{p-\gamma+\alpha-1}(t) F^{1-\alpha}(t)
t^{\alpha-1} dt.
\end{multline}
\end{remark}

\begin{remark}
If we take $\alpha=1$ in inequality \eqref{eqcor133}, then one gets
\begin{equation*}
\int_{0}^{\infty}\frac{g(t)\Big(\int_{t}^{\infty}
g(s)f(s)ds\Big)^p}{\Big(\int_{t}^{\infty}g(s)ds\Big)^\gamma}dt
\leq\Big(\frac{p}{\gamma-1}\Big)^p\int_{0}^{\infty}g(t)f^p(t)
\Big(\int_{t}^{\infty}g(s)f(s)ds\Big)^{p-\gamma}dt,
\end{equation*}
which is the continuous analogous of Bennett's inequality \eqref{h31}.
\end{remark}

\begin{corollary}
\label{cor14}
Choosing $\mathbb{T}=h\mathbb{Z}$ in Theorem~\ref{thm4},  
inequality \eqref{eq4} gives that
\begin{multline}
\sum_{t=\frac{a}{h}}^{\infty}k(ht)v(ht)w(ht+h)g(ht)H^{-\gamma}(ht)F^p(ht)t^{\alpha-1}\\
\leq\Big(\frac{p+\beta}{\gamma-\theta-1}\Big)^p
\sum_{t=\frac{a}{h}}^{\infty}\frac{k(ht)v(ht)w(ht+h)r^p(ht)f^p(ht)H^{\gamma(p-1)}(ht)}
{g^{p-1}(t)H^{p(\gamma-1)}(ht+h)}t^{\alpha-1},
\end{multline}
where
\begin{equation*}
H(t)=h^{\alpha}\sum_{s=\frac{t}{h}}^{\infty}g(hs)s^{\alpha-1} 
\quad \text{and} \quad  
F(t)=h^{\alpha}\sum_{s=\frac{t}{h}}^{\infty}r(hs)f(hs)s^{\alpha-1}.
\end{equation*}
\end{corollary}

\begin{corollary}
\label{cor15}
In the particular case $h=1$, that is, in the discrete 
time scale $\mathbb{T}=\mathbb{Z}$, Corollary~\ref{cor14} reduces to
\begin{multline}
\label{eqcor15}
\sum_{t=a}^{\infty}k(t)v(t)w(t+1)g(t)H^{-\gamma}(t)F^p(t)t^{\alpha-1}\\
\leq\Big(\frac{p+\beta}{\gamma-\theta-1}\Big)^p\sum_{t=a}^{\infty}
\frac{k(t)v(t)w(t+1)r^p(t)f^p(t) H^{\gamma(p-1)}(t)}{g^{p-1}(t)
H^{p(\gamma-1)}(t+1)}t^{\alpha-1},
\end{multline}
where
\begin{equation*}
H(t)=\sum_{s=t}^{\infty}g(s)s^{\alpha-1} 
\quad \text{and} \quad  
F(t)=\sum_{s=t}^{\infty}f(s)r(s)s^{\alpha-1}.
\end{equation*}
\end{corollary}

\begin{remark}
With $S_{16}$, $S_{17}$, $S_{19}$ and $S_{21}$,
then \eqref{eqcor15} of Corollary \ref{cor15}
gives a different form of inequality \eqref{h31}.
\end{remark}

\begin{corollary}
\label{cor16}
If $\mathbb{T}=\overline{q^\mathbb{Z}}$ in Theorem~\ref{thm4}, 
it follows from \eqref{q^z} and \eqref{eq4} that
\begin{multline}
\sum_{t=\log_q a}^{\infty}k(q^t)v(q^t)w(q^{t+1})
g(q^t)H^{-\gamma}(q^t)F^p(q^t)q^{\alpha t}\\
\leq\Big(\frac{p+\beta}{\gamma-\theta-1}\Big)^p
\sum_{t=\log_q a}^{\infty}\frac{k(q^t)v(q^t)w(q^{t+1})r^p(q^t)f^p(q^t)
H^{\gamma(p-1)}(q^t)}{g^{p-1}(q^t)H^{p(\gamma-1)}(q^{t+1})}q^{\alpha t},
\end{multline}
where
\begin{equation*}
H(t)=(q-1)\sum_{s=\log_q t}^{\infty}g(q^s)q^{\alpha s} 
\quad \text{and} \quad  
F(t)=(q-1)\sum_{s=\log_q t}^{\infty} r(q^s)f(q^s)q^{\alpha s}.
\end{equation*}
\end{corollary}


\section{Conclusion}
\label{sec:4}

Hardy-type inequalities have many applications
and are subject to strong research: see the books
\cite{MR3969969,MR3676556,MR3561405,MR1982932}
and the recent publications \cite{MR4414114,MR4410837,MR4409663}.
In this manuscript, by employing the 
$\alpha$-conformable fractional calculus on time scales
of Benkhettou et al. \cite{sdm31}, several new 
Hardy-type inequalities were proved. The results extend several 
dynamic inequalities known in the literature, being new even
in the discrete, continuous and quantum settings.


\subsection*{Acknowledgments}

Torres was supported by FCT under project UIDB/04106/2020 (CIDMA).


\subsection*{Data Availability Statement}

The authors declare that all data supporting 
the findings of this study are available within the article.


\subsection*{Conflict of Interest Statement}

The Authors declare that there is no conflict of interest.


\bibliographystyle{abbrv}
\bibliography{El-Deeb_Makharesh_Torres}


\end{document}